\documentclass[10pt, a4paper, english]{smfart}
\usepackage{amssymb}
\usepackage{tikz-cd}
\usepackage{mathtools}
\usepackage{hyperref}
\usepackage{enumitem}
\usepackage{comment}

\renewcommand{\geq}{\geqslant}
\renewcommand{\leq}{\leqslant}

\theoremstyle{plain}
\newtheorem{theorem}[subsection]{Theorem}
\newtheorem{lemma}[subsection]{Lemma}
\newtheorem{conjecture}[subsection]{Conjecture}
\newtheorem{corollary}[subsection]{Corollary}
\newtheorem{proposition}[subsection]{Proposition}

\theoremstyle{remark}
\newtheorem{remark}[subsection]{Remark}
\newtheorem{definition}[subsection]{Definition}

\renewcommand{\AA}{\ensuremath{\mathbb{A}}}
\newcommand{\CC}{\ensuremath{\mathbb{C}}}

\newcommand{\PP}{\ensuremath{\mathbb{P}}}
\newcommand{\QQ}{\ensuremath{\mathbb{Q}}}

\newcommand{\ZZ}{\ensuremath{\mathbb{Z}}}

\DeclareMathOperator{\pr}{pr}
\DeclareMathOperator{\Aut}{Aut}
\DeclareMathOperator{\End}{End}
\DeclareMathOperator{\h}{\mathfrak{h}} 
\DeclareMathOperator{\NS}{\mathrm{NS}} 
\DeclareMathOperator{\Km}{\mathrm{Km}} 
\DeclareMathOperator{\Pic}{\mathrm{Pic}}

\title{Sixfolds of generalized Kummer type and K3 surfaces}
\author{Salvatore Floccari}
\address{Institute of Algebraic Geometry, Leibniz University Hannover, Germany}
\email{floccari@math.uni-hannover.de}

\begin{document}
	
	
		\begin{abstract}
			We prove that any hyper-K\"{a}hler sixfold $K$ of generalized Kummer type has a naturally associated manifold $Y_K$ of $\mathrm{K}3^{[3]}$-type. It is obtained as crepant resolution of the quotient of $K$ by a group of symplectic involutions acting trivially on its second cohomology. When $K$ is projective, the variety $Y_K$ is birational to a moduli space of stable sheaves on a uniquely determined projective~$\mathrm{K}3$ surface~$S_K$. As application of this construction we show that the Kuga-Satake correspondence is algebraic for the K3 surfaces $S_K$, producing infinitely many new families of $\mathrm{K}3$ surfaces of general Picard rank $16$ satisfying the Kuga-Satake Hodge conjecture.
		\end{abstract}
	
	\maketitle
	
	\section{Introduction}
	
Together with manifolds of $\mathrm{K}3^{[n]}$-type, deformations of generalized Kummer varieties constitute the most well studied hyper-K\"{a}hler manifolds. We refer to this deformation type in dimension $2n$ as to the $\mathrm{Kum}^n$-type. After Beauville~\cite{beauville1983varietes} gave the first examples of such hyper-K\"{a}hler manifolds, many more have been constructed from moduli spaces of stable sheaves on abelian surfaces, see \cite{Yos01}. 
However, the varieties so obtained always have Picard rank at least $2$, and our understanding of a general projective variety of $\mathrm{Kum}^n$-type remains poor. In fact, for the time being, no construction of such variety is known (but see \cite{O'G23} for some recent ideas).

In the present article we partially remedy this for hyper-K\"{a}hler sixfolds of generalized Kummer type, by associating to any $K$ of $\mathrm{Kum}^3$-type a hyper-K\"{a}hler manifold~$Y_K$ of~$\mathrm{K}3^{[3]}$-type. They are related by a dominant rational map 
\[
K\dashrightarrow Y_K
\]
of degree $2^5$, described as follows.

Any $K$ of $\mathrm{Kum}^3$-type admits an action of the group $(\ZZ/4\ZZ)^4\rtimes \ZZ/2\ZZ$ by symplectic automorphisms, where $\ZZ/2\ZZ$ acts on the first factor as $-1$. In fact, this is the group $\Aut_0(K)$ of automorphisms of $K$ which act trivially on its second cohomology, which is deformation invariant by \cite{hassettTschinkel}. It has been computed in \cite{boissiere2011higher} for the generalized Kummer variety associated to an abelian surface. We let $G\subset \Aut_0(K)$ be the subgroup generated by the automorphisms whose fixed locus contains a $4$-dimensional component. We will show in Lemma \ref{lem:G} that \[G\cong (\ZZ/2\ZZ)^5.\]

\begin{theorem}\label{thm:MainResult}
	Let $K$ be a manifold of $\mathrm{Kum}^3$-type. The quotient $K/G$ admits a resolution $Y_K\to K/G$	with $Y_K$ a manifold of $\mathrm{K}3^{[3]}$-type. 
\end{theorem} 

The resolution $Y_K\to K/G$ is obtained via a single blow-up of the reduced singular locus. It can also be described as the quotient by $G$ of the blow-up of $K$ at the union of the fixed loci of all non-trivial automorphisms in $G$. We show that there is an isometry of transcendental Hodge structures
	$$ H_{\mathrm{tr}}^2(Y_K, \QQ) \xrightarrow{\ \sim \ } H_{\mathrm{tr}}^2(K, \QQ)(2),$$
	where on the right hand side the form is multiplied by $2$. This parallels the classical construction of the Kummer $\mathrm{K}3$ surface $\Km(A)$ associated to an abelian surface $A$, where $H^2_{\mathrm{tr}}(\Km(A), \ZZ)$ is Hodge isometric to $H^2_{\mathrm{tr}}(A,\ZZ)(2)$.
 The study of the integral transcendental lattices of the manifolds $Y_K$ will be the subject of future work.

 As a consequence, we obtain a well-defined $\mathrm{K}3$ surface associated to any projective sixfold of generalized Kummer type.

\begin{theorem}\label{thm:AssociatedK3}
	 Let $K$ be a projective variety of $\mathrm{Kum}^3$-type. There exists a unique (up to isomorphism) projective $\mathrm{K}3$ surface $S_K$ such that the variety $Y_K$ of $\mathrm{K}3^{[3]}$-type given by Theorem \ref{thm:MainResult} is birational to a moduli space $M_{S_K, H}(v)$ of stable sheaves on~$S_K$, for some primitive Mukai vector $v$ and a $v$-generic polarization $H$.
\end{theorem}
	
	 We call $S_K$ the $\mathrm{K}3$ surface associated to the sixfold $K$. It is characterized by the existence of a Hodge isometry $H^2_{\mathrm{tr}}(S_K, \ZZ)\xrightarrow{\ \sim \ } H^2_{\mathrm{tr}}(Y_K,\ZZ)$. These surfaces come in countably many $4$-dimensional families of general Picard rank $16$; up to isogeny, they are the K3 surfaces $S$ admitting an isometric embedding $H^2_{\mathrm{tr}}(S,\QQ)\hookrightarrow \Lambda_{\mathrm{Kum}^3}(2)\otimes_{\ZZ} \QQ$ of rational quadratic spaces. Here, $\Lambda_{\mathrm{Kum}^3}$ is the lattice which is the second cohomology of manifolds of $\mathrm{Kum}^3$-type; it was computed in \cite{beauville1983varietes} that $\Lambda_{\mathrm{Kum}^3}=U^{\oplus 3}\oplus \langle -8 \rangle$, where we denote by $U$ the hyperbolic plane.

	\subsection*{Applications}
	In a series of papers \cite{MarkmanBeauvilleBogomolov}, \cite{markman2019monodromy}, \cite{markmanRational}, Markman has proven striking results on the Hodge conjecture for hyper-K\"{a}hler varieties of $\mathrm{Kum}^n$ and $\mathrm{K}3^{[n]}$-type. He uses Verbitsky's theory of hyperholomorphic sheaves \cite{verbitsky1997} to produce very interesting algebraic cycles on general projective hyper-K\"{a}hler varieties via deformation. Through our construction, we are able to deduce more cases of the Hodge conjecture from his results.
	
	Our main application is to the Kuga-Satake Hodge conjecture for $\mathrm{K}3$ surfaces~\cite{vanGeemen}. The Kuga-Satake construction \cite{deligne1971conjecture} associates via Hodge theory an abelian variety $\mathrm{KS}(S)$ to any projective K3 surface $S$, and the conjecture  
	predicts the existence of an algebraic cycle inducing an embedding of the transcendental Hodge structure of the surface into the second cohomology of $\mathrm{KS}(S)\times \mathrm{KS}(S)$. 
	 It is known to hold in many cases for $\mathrm{K}3$ surfaces of Picard number at least $17$ (\cite{morrison1985}), but it is wide open otherwise. There are a couple of families of $\mathrm{K}3$ surfaces of general Picard rank~$16$ for which the conjecture is known, namely, the family of double covers of the plane branched at $6$ lines \cite{paranjape} and the family of K3 surfaces with $15$ nodes in $\PP^4$ \cite{ILP}.
	 
	Theorem \ref{thm:application1} below gives infinitely many new families of K3 surfaces of general Picard rank $16$ for which the Kuga-Satake Hodge conjecture holds true.
	
	\begin{theorem}\label{thm:application1}
		Let $S$ be a projective $\mathrm{K}3$ surface such that there exists an isometric embedding of $H^2_{\mathrm{tr}}(S,\QQ)$ into $\Lambda_{\mathrm{Kum}^3}(2)\otimes_{\ZZ}\QQ$, 
		where $(2)$ indicates that the form is multiplied by $2$. 
		Then, there exists an algebraic cycle $\gamma$ on $S\times \mathrm{KS}(S)\times \mathrm{KS}(S)$ inducing an embedding $\gamma_*\colon H^2_{\mathrm{tr}}(S,\QQ)\hookrightarrow H^2(\mathrm{KS}(S)\times \mathrm{KS}(S), \QQ)$
		of rational Hodge structures.
	\end{theorem}
	
	Via Theorem \ref{thm:AssociatedK3}, we deduce this statement from the validity of the Kuga-Satake Hodge conjecture for varieties of $\mathrm{Kum}^n$-type, established by Voisin \cite{voisinfootnotes} as a consequence of results of O'Grady \cite{O'G21} and Markman \cite{markman2019monodromy}.
	Once Theorem \ref{thm:application1} is proven, a result of Varesco \cite{varesco} implies that the Hodge conjecture holds for all powers of any K3 surface satisfying the condition of the above theorem, see Corollary~\ref{cor:HCpowers}.
		
	In the recent preprint \cite{markmanRational}, Markman proves that any rational Hodge isometry $H^2(X,\QQ)\xrightarrow{ \ \sim \ } H^2(X',\QQ)$ between varieties of $\mathrm{K}3^{[n]}$-type is algebraic. He expects that an extension of his argument will lead to the analogous result for varieties of~$\mathrm{Kum}^n$-type. In dimension $6$, we can obtain it from the $\mathrm{K}3^{[3]}$ case via Theorem \ref{thm:MainResult}. 
	
	\begin{theorem}\label{thm:application2}
		Let $K, K'$ be projective varieties of deformation type $\mathrm{Kum}^3$. Let $f\colon H^2(K, \QQ)\xrightarrow{ \ \sim \ } H^2(K', \QQ)$ be a Hodge isometry. Then $f$ is induced by an algebraic correspondence.
	\end{theorem}

	\subsection*{Overview of the contents}
	
	We sketch the proof of Theorem \ref{thm:MainResult}.
	First, we calculate the fixed loci $K^g$ of all the automorphims $g\in G$: we show that $\bigcup_{g\neq 1 \in G} K^g$ is the union of $16$ hyper-K\"{a}hler manifolds of $\mathrm{K}3^{[2]}$-type, and we determine the various intersections of these components. Thanks to the deformation invariance of $G$, it is in fact enough to calculate these loci in the special case of the generalized Kummer sixfold on an abelian surface. 
	
	We then describe explicitly the singularities of $K/G$. It turns out that these are of a particular simple nature, modeled on products of ordinary double points on surfaces. A resolution $Y_K$ of $K/G$ is obtained via a single blow-up of the singular locus. The quotient $K/G$ is a primitive symplectic orbifold as studied by Fujiki~\cite{fujiki1983} and Menet~\cite{menet}, and we use a criterion due to Fujiki to prove that $Y_K$ is a hyper-K\"ahler manifold. 
	Moreover, the resolution can be performed in families, so that the~$Y_K$ are deformation equivalent to each other. 
	 To complete the proof of Theorem~\ref{thm:MainResult} it is therefore sufficient to find a single~$K$ of $\mathrm{Kum}^3$-type such that $Y_K$ is of $\mathrm{K}3^{[3]}$-type.
		
	The specific example we study is a Beauville-Mukai system $K_J(v_3)$ on a general principally polarized abelian surface $\Theta\subset J$. It admits a Lagrangian fibration to the linear system $|2\Theta|=\PP^3$, whose general fibres parametrize certain degree $3$ line bundles supported on curves in the linear system. We show that the norm map for line bundles induces a dominant rational map of degree $2^5$ from $K_J(v_3)$ onto a Beauville-Mukai system $M_{\Km(J)}(w_3)$, a moduli space of sheaves on the Kummer $\mathrm{K}3$ surface $\Km(J)$ associated to~$J$; the hyper-K\"{a}hler variety $M_{\Km(J)}(w_3)$ is birational to ${\Km(J)}^{[3]}$. 
	We next prove that the norm map descends to a birational~map
	\[
	K_J(v_3)/G \dashrightarrow M_{\Km(J)}(w_3).
	\] 
 	It follows that the hyper-K\"ahler manifold $Y_{K_J(v_3)}$ is birational to $\mathrm{Km}(J)^{[3]}$, and hence~$Y_{K_J(v_3)}$ is of $\mathrm{K}3^{[3]}$-type by Huybrechts' theorem \cite{Huy99} that birational hyper-K\"{a}hler manifolds are deformation equivalent. 
 
 Once Theorem \ref{thm:MainResult} is proven, the other results follow from it. When $K$ is projective, we define the associated $\mathrm{K}3$ surface $S_K$ as the unique (up to isomorphism) $\mathrm{K}3$ surface whose transcendental lattice is Hodge isometric to $H^2_{\mathrm{tr}}(Y_K, \ZZ)$. 
 This is justified by the fact that $H^2_{\mathrm{tr}}(Y_K, \ZZ)$ has rank at most $6$ and hence it appears as transcendental lattice of some $\mathrm{K}3$ surface of Picard rank at least $16$. Moreover, for such Picard numbers, two $\mathrm{K}3$ surfaces with Hodge isometric transcendental lattices are isomorphic, and hence $S_K$ is determined up to isomorphism. 
 To prove our applications, we exploit the algebraic cycle in $K\times Y_K$ given by the rational map $K\dashrightarrow Y_K$.

The organization of this text is as follows. In Section \ref{sec:FixedLoci} we calculate the fixed locus of the automorphisms of $G$. In Section \ref{sec:specialCase} we study in detail the case of $K_J(v_3)$ as outlined above. In Section \ref{sec:AssociatedK3} we prove Theorems \ref{thm:MainResult} and \ref{thm:AssociatedK3}.
Finally, in Section \ref{sec:KSHodge} we use our construction to prove Theorems \ref{thm:application1} and \ref{thm:application2}.
		
\subsection*{Aknowledgements}
It is a pleasure to thank Ben Moonen, Kieran O'Grady, Claudio Pedrini and Stefan Schreieder for their interest in this project. I am grateful to Lie Fu and Giovanni Mongardi for their comments on a first version of this article, and to Gr\'egoire Menet for useful discussions. I thank Mauro Varesco for suggesting to me Corollary 5.8. I thank Giacomo Mezzedimi and Domenico Valloni for many conversations related to these topics, and the anonymous referee for his/her comments. 
		
	\section{Automorphisms trivial on the second cohomology}\label{sec:FixedLoci}

	Let $K$ be a manifold of $\mathrm{Kum}^3$-type and let $\Aut_0(K)$ be the group of automorphisms of $K$ which act trivially on $H^2(K,\ZZ)$. We let $K^h$ denote the fixed locus of an automorphism $h$ of $K$.
	\begin{definition}\label{def:G}
		We denote by $G$ the subgroup of $\Aut_0(K)$ generated by the automorphisms $g\in\Aut_0(K)$ such that $K^g$ has a $4$-dimensional component.
	\end{definition}
	
	We will see in Lemma \ref{lem:G} below that $G\cong (\ZZ/2\ZZ)^5$. The main result proven in this section is then the following. 
	
	\begin{theorem}\label{thm:fixedLoci}
		Let $Z\subset K$ be defined as $Z\coloneqq \bigcup_{g\neq 1\in G} K^{g}$. Then $Z$ is the union of $16$ fourfolds $Z_j$, each of which is a smooth hyper-K\"{a}hler manifold of $\mathrm{K}3^{[2]}$-type. Moreover, two distinct components intersect in a K3 surface, three distinct components intersect in $4$ points, and four or more distinct components do not intersect.
	\end{theorem}
	
	In order to prove this, we will compute the fixed loci of all automorphisms in~$G$. 
	Thanks to the deformation invariance of automorphisms trivial on the second cohomology it will be sufficient to treat the case of the generalized Kummer variety associated to an abelian surface $A$, i.\ e.\ $K$ is the fibre over $0$ of the composition
	\[
	A^{[4]} \xrightarrow{ \ \nu \ } A^{(4)} \xrightarrow{ \ \sum \ } A
	\]
	of the Hilbert-Chow morphism with the summation map. We denote by $A^{(4)}_0\subset A^{(4)}$ the fibre over $0$ of $\sum$. The restriction $\nu \colon K\to A^{(4)}_0$ is a crepant resolution. 
	
	In this case the group of automorphisms of $K$ acting trivially on the second cohomology has the natural description  
	$
	\Aut_0(K) = A_4 \rtimes \langle -1 \rangle,
	$
	and the action on $K$ and~$A^{[4]}$ is induced by that on $A$, see \cite{boissiere2011higher}.
	
	\subsection{Notation}
	We let $A_n$ be the group of points of order $n$ of $A$.
	Via the isomorphism above, we write the elements of $\Aut_0(K^3(A))$ as $(\epsilon, \pm 1)$ for $\epsilon \in A_4$. For $\tau\in A_2$, we denote by $A_{2,\tau}$ the set $\{a\in A \ | \ 2a=\tau\}$, which consists of $16$ points. The quotient surface $A/\langle (\tau, -1) \rangle$ has $16$ nodes corresponding to points in $A_{2,\tau}$;
	its minimal resolution~$\Km^{\tau}(A)$ is isomorphic to the Kummer $\mathrm{K}3$ surface associated to~$A$.
	
	\begin{lemma}\label{lem:G}
	Let $K$ be any manifold of $\mathrm{Kum}^3$-type. Then $G\cong (\ZZ/2\ZZ)^5$.
	\end{lemma} 
	\begin{proof}
		Since automorphisms in $\Aut_0(K)$ deform with $K$, their fixed loci deform as well. Therefore, it suffices to prove the Lemma for the generalized Kummer variety $K=K^3(A)$ on an abelian surface $A$. In this case we show that
		\[
		G=A_2\times \langle -1\rangle \cong (\ZZ/2\ZZ)^5.
		\]
		
		We have $\Aut_0(K^3(A))= A_4\rtimes \langle -1\rangle$. If $h\in\Aut_0(K^3(A))$ is induced by translation by $0\neq \epsilon\in A_4$, then, by \cite[Lemma 3.5]{oguiso2020no}, the fixed locus $(K^3(A))^h$ is a union of K3 surfaces if $2\epsilon=0$ and consists of isolated point if $2\epsilon\neq 0$. 
		Consider now $h=(\epsilon, -1)$. 
		The fixed locus of $h=(\epsilon, -1)$ on $A^{(4)}$ consists~of
		\[
		W_{\epsilon} = \{ (a, -a -\epsilon , b , - b - \epsilon), \ \text{ for } a,b \in A \} \subset A^{(4)},
		\]
		and subvarieties of lower dimension. 
		If $\epsilon\in A_2$, then $W_{\epsilon}$ is contained in $A^{(4)}_0$, and hence $(\epsilon, -1)\in G$. 
		If $\epsilon \notin A_2$, then $W_{\epsilon}$ has empty intersection with $A^{(4)}_0$; the action of~$h$ on~$A^{(4)}_0$ fixes isolated points and the surfaces $\{ (\alpha , b, -b-\epsilon, - \alpha + \epsilon),\ b\in A\}$, where~$\alpha\in A$ satisfies $\alpha=-\alpha-\epsilon$. The general point of such a surface is supported on~$4$ distinct point. Therefore, the fixed locus $(K^3(A))^h$ consists of surfaces and isolated points, because it is a union of symplectic varieties and the fibres of $\nu$ have dimension at most $3$.  
		We conclude that $G$ is generated by the $16$ involutions $(\tau,-1)$ with $\tau\in A_2$, and hence $G=A_2\times \langle -1\rangle$.
	\end{proof}
	
	\subsection{Relevant subvarieties} 
	In what follows we let $K=K^3(A)$ be the generalized Kummer variety on an abelian surface $A$. We will calculate the fixed locus of automorphisms in $G=A_2\times \langle -1\rangle$ acting on $K$. The next definition introduces the relevant subvarieties. We denote by $\xi=(a,b,c,d)$ a point of $A_0^{(4)}$, and refer to $\{a,b,c,d\}$ as the support of $\xi$.
	\begin{definition} \label{def:RelevantLoci}
		We define the following subvarieties of $A_0^{(4)}$:
		 \begin{itemize}
		 	\item for any $\tau\in A_{2}$, we let \[W_{\tau}\coloneqq \{(a, b, -a+\tau, -b+\tau), \ \text{ for } \ a,b\in A \};
		 	\]
		 	\item for a pair $(\tau, \theta)\in A_2 \times A_2$ with $\tau\neq 0$, we let 
		 	\[
		 	V_{\tau, \theta} \coloneqq \{ (a, a+\tau, -a+\theta , -a+\tau+\theta), \ \text{ for } \ a\in A\}.
		 	\]
		 \end{itemize} 
	 We denote by $\overline{W_{\tau}}$ and $\overline{V_{\tau, \theta}}$ the subvarieties of $K$ obtained as strict transform of $W_{\tau}$ and $V_{\tau, \theta}$ under the Hilbert-Chow morphism, respectively.
	\end{definition} 

Note that none among the $W_{\tau}$ and $V_{\tau, \theta}$ is contained in the exceptional locus of $\nu\colon K\to A^{(4)}_0$, so that $\overline{W_{\tau}}$ and $\overline{V_{\tau,\theta}}$ are well-defined. We will show in Lemma \ref{lem:GeometricDescription} that the $\overline{W_{\tau}}$ are hyper-K\"ahler varieties of $\mathrm{K}3^{[2]}$-type and that the $\overline{V_{\tau,\theta}}$ are K3 surfaces.

\begin{remark} \phantomsection \label{rmk:DescriptionRelevantLoci}
\begin{enumerate}[label=(\roman*)]
	\item For $\tau \in A_2$, the $G$-equivariant morphism $\psi_{\tau}\colon A\times A\to A^{(4)}$ such that $(a,b)\mapsto (a, b, -a+\tau, -b+\tau)$ has degree $8$ and induces an isomorphism $$(A/ \langle (\tau, -1) \rangle)^{(2)} \xrightarrow{\ \sim \ } W_{\tau}.$$
	Moreover, $W_{\tau}\neq W_{\tau'}$ unless $\tau=\tau'$, as can be seen considering a point $(\epsilon, \epsilon, \epsilon, \epsilon)$ for some $\epsilon\in A_{2,\tau}$.
	\item Given $0\neq \tau \in A_2$ and $ \theta \in A_2$, the $G$-equivariant morphism $\psi_{\tau, \theta}\colon A\to A^{(4)}$ defined by $a \mapsto (a, a+\tau, -a+\theta, -a+\tau+\theta)$ induces an isomorphism
	$$ A/\langle \tau, (\theta, -1) \rangle \xrightarrow{\ \sim \ } V_{\tau, \theta}.$$ 
	For a given $\tau\in A_2$, the union $\bigcup_{\theta\in A_2} V_{\tau, \theta}$ has $8$ irreducible components. In fact, $V_{\tau, \theta}=V_{\tau, \theta'}$ if and only if $\theta' \in \{\theta, \theta+\tau\}$, which can be seen considering a point $(a, a+\tau, a, a+\tau)$: it belongs to $V_{\tau, \theta}$ if and only if $a\in A_{2, \theta}$ or $a\in A_{2,\theta+\tau}$. Moreover, it can be easily checked that $V_{\tau, \theta}\neq V_{\tau', \theta'}$ for $\tau\neq \tau'$. In total, we therefore have $120$ distinct surfaces $V_{\tau, \theta}$.
\end{enumerate}	
\end{remark} 

We recall a result of Kamenova-Mongardi-Oblomonkov which will be used in the sequel.

\begin{proposition}[{\cite[Lemma B.1]{KMO}}]\label{lem:local}
	Let $\iota\colon \CC^2\to \CC^2$ be the involution given by $(x,y)\mapsto (-x, -y)$. Consider the rational map
	$$\Phi\colon (\mathrm{Bl}_{0}(\CC^2/\iota))^{[n]} \dashrightarrow (\CC^2)^{[2n]} $$
	which maps a general $\zeta={(P_1, \dots, P_{n})}$ to $\xi={(P_1, \iota(P_1), \dots, P_n, \iota(P_n))}$. Then $\Phi$ extends to a closed embedding, with image the fixed locus for the action of $\iota$ on $(\CC^2)^{[2n]}$.	
\end{proposition} 

The next lemma describes the varieties $\overline{V_{\tau, \theta}}$ and $\overline{W_{\tau}}$ introduced above.

\begin{lemma} \phantomsection \label{lem:GeometricDescription}
	\begin{enumerate}[label=(\roman*)]
		\item For $\tau\neq 0$, $ \theta\in A_2$, the rational map $\phi_{\tau,\theta}$ defined by the commutative diagram
		\[
		\begin{tikzcd}
				\Km^{\theta} (A/\langle \tau \rangle) \arrow[dashed]{r}{\phi_{\tau,\theta}} \arrow{d}{q} & K \arrow{d}{\nu}  \\
				A/\langle\tau, (\theta, -1) \rangle \arrow{r}{\psi_{\tau,\theta}} & A^{(4)}_0
		\end{tikzcd}
		\]
		where $q$ is the natural map, extends to a regular morphism, which is an isomorphism onto its image $\overline{V_{\tau, \theta}}$.
		\item For $\tau\in A_2$, the rational map $\phi_{\tau}$ defined by the commutative diagram
		\[
		\begin{tikzcd}
		(\Km^{\tau}(A))^{[2]} \arrow[dashed]{r}{\phi_{\tau}} \arrow{d}{q} & K \arrow{d}{\nu}\\
		(A/\langle (\tau, -1)\rangle )^{(2)} \arrow{r}{\psi_{\tau}} & A^{(4)}_0
		\end{tikzcd}
		\] 
	where $q$ is the composition of the Hilbert-Chow morphism with the natural map $(\Km^{\tau}(A))^{(2)} \to (A/\langle (\tau, -1)\rangle )^{(2)}$, extends to a regular morphism, which is an isomorphism onto its image $\overline{W_{\tau}}$. 
	\end{enumerate}
\end{lemma}
\begin{proof}
	We prove $(i)$. As translations are fixed point free, for any $\alpha \in A/\langle \tau, (\theta, -1)\rangle$ the support of $\psi_{\tau, \theta} (\alpha)\in V_{\tau, \theta}$ consists of either~$2$ or $4$ distinct points, according to whether $\alpha$ is a node or not. If $\alpha$ is not a node $q^{-1}(\alpha)$ is a single point at which~$\phi_{\tau, \theta}$ is well-defined.
	Otherwise, $\psi_{\tau, \theta} (\alpha) =(a, a+\tau, a, a+\tau)$ for some $a$ in $A_{2, \theta}$ or $A_{2,\tau+\theta}$. 	
	Then, there is a canonical identification $\nu^{-1}(\psi_{\tau, \theta} (\alpha))=\PP(T_aA) \times \PP(T_{a+\tau} A)$. The exceptional divisor in $\Km^{\theta}(A/\langle \tau\rangle)$ corresponding to the node $\alpha$ is identified with $\PP(T_a(A))$. Translation by $\tau$ gives an isomorphism $\PP(T_a(A))\xrightarrow{\ \sim \ } \PP(T_{a+\tau})$, and the morphism $\phi_{\tau, \theta}$ is extended via $\phi_{\tau, \theta} (t)\coloneqq (t, \tau(t))$ for $t\in \PP(T_a(A))$. 
	
	The proof of $(ii)$ is similar. 
	If $\psi_{\tau} (\alpha, \beta)$ consists of~$4$ distinct points then $q^{-1}(\alpha, \beta)$ is a single point at which $\phi_{\tau}$ is well-defined. 
	If the support of $\psi_{\tau}(\alpha, \beta)$ consists of~$3$ distinct points, then exactly one between $\alpha$ and $\beta$ is a node of $A/\langle (\tau, -1) \rangle$, and $\psi_{\tau}(\alpha, \beta)=(a, \epsilon, -a+\tau, \epsilon)$ for some $a\in A\setminus A_{2,\tau}$ and $\epsilon\in A_{2,\tau}$; we extend $\phi_{\tau}$ via the canonical identifications $q^{-1}(\alpha, \beta) = \PP(T_{\epsilon}(A))=\nu^{-1}(\psi_{\tau}(\alpha,\beta))$.   
	If the support of $\psi_{\tau}(\alpha, \beta)$ consists of $2$ distinct points then either $\alpha=\beta$ for some smooth point or $\alpha\neq \beta$ with both $\alpha$ and $\beta$ nodes of $A/\langle(\tau,-1)\rangle$. 
	In the first case there exists $a\in A\setminus A_{2,\tau}$ such that $\psi_{\tau}(\alpha,\beta)=(a, a, -a+\tau, -a+\tau)$. Then $q^{-1}(\alpha, \beta) = \mathbb{P}(T_a A)$ and $\nu^{-1}(\psi_{\tau} (\alpha, \beta))= \PP(T_{a}A)\times \PP(T_{-a+\tau} A)$, and we define~$\phi_{\tau}$ by $t\mapsto (t, (\tau, -1)(t))$. 
	In the second case $\psi_{\tau}(\alpha, \beta)= (\epsilon_1, \epsilon_2, \epsilon_1, \epsilon_2)$ for some $\epsilon_1\neq \epsilon_2$ both in $A_{2,\tau}$, and we extend $\phi_{\tau}$ via the canonical isomorphisms $q^{-1}(\alpha, \beta) = \PP(T_{\epsilon_1}A)\times \PP(T_{\epsilon_2}A) = \nu^{-1}(\psi_{\tau}(\alpha,\beta))$.
	Finally, consider the case when $\alpha=\beta$ and $\alpha$ is a node.
	Locally analytically around the node $\alpha$ of $A/\langle (\tau, -1)\rangle$, the morphism $q$ is isomorphic to the composition 
	\[
	 (\mathrm{Bl}_0(\CC^2/\iota))^{[2]}\to   (\mathrm{Bl}_0(\CC^2/\iota))^{(2)} \to (\CC^2/\iota)^{(2)},
	\]
	where $\iota(x,y)=(-x,-y)$ is the restriction of the involution $(\tau, -1)$. The rational map $\phi_{\tau}$ is identified with $\Phi \colon (\mathrm{Bl}_0(\CC^2/\iota))^{[2]} \dashrightarrow (\CC^2)^{[4]}$ which maps a general $\zeta=(\alpha,\beta)$ to the subscheme $\xi=(\alpha,\iota(\alpha), \beta, \iota(\beta))$. Then Proposition \ref{lem:local} implies that $\phi_\tau$ extends to a closed embedding $(\mathrm{Bl}_0(\CC^2/\iota))^{[2]}\hookrightarrow (\CC^2)^{[4]}$, completing the proof.
\end{proof}	

The various intersections of the submanifolds $\overline{W_{\tau}}\subset K$ are as follows.
\begin{lemma}\phantomsection
	\label{lem:intersections}
	\begin{enumerate}[label=(\roman*)]
		\item For $\tau_1 \neq \tau_2$ in $A_2$ we have $\overline{W_{\tau_1}}\cap \overline{W_{\tau_2}}=\overline{V_{\tau_1+\tau_2, \tau_1}}$.
		\item For pairwise distinct $\tau_1$, $\tau_2$, $\tau_3$ in $A_2$, the intersection $\overline{W_{\tau_1}}\cap \overline{W_{\tau_2}}\cap \overline{W_{\tau_3}}$ consists of the $4$ distinct points of the set
		\[
		\nu^{-1}( \{ (a, a+\tau_1+\tau_2, -a+\tau_1, -a+\tau_2) \in A_0^{(4)}, \ a\in A_{2, \tau_1+\tau_2+\tau_3}\}) \subset K.
		\]
		\item The intersection of $4$ or more distinct submanifolds $\overline{W_{\tau}}$ is empty.
	\end{enumerate} 
\end{lemma} 
\begin{proof}
	$(i)$. A point $\xi\in A_0^{(4)}$ belongs to ${W_{\tau_1}}\cap {W_{\tau_2}}$ if and only if it can be written as $\xi=(a,b,-a+\tau_1, -b+\tau_1)$ as well as $\xi=(a', b' , -a'+\tau_2, -b'+\tau_2)$, for some $a,b,a',b'$ in~$A$. Then, we may assume that $a'=a$; since $\tau_1\neq \tau_2$, it follows that $-a+\tau_1 \neq -a'+\tau_2$. This forces either $b' =-a+\tau_1$ or $-b'+\tau_2= -a +\tau_1$; in both cases, we can write
	\[\xi = (a, -a+\tau_1, -a+\tau_2, a+\tau_1+\tau_2),\] 
	i.\ e.\ $\xi\in V_{\tau_1+\tau_2, \tau_1}$. Thus ${W_{\tau_1}}\cap {W_{\tau_2}}=V_{\tau_1+\tau_2, \tau_1}$, and hence for their strict transforms we have $\overline{W_{\tau_1}}\cap \overline{W_{\tau_2}}=\overline{V_{\tau_1+\tau_2, \tau_1}}$.
	
	$(ii)$.
	By $(i)$, a point $\xi$ lies in the intersection $W_{\tau_1}\cap W_{\tau_2}\cap W_{\tau_2}$ if and only if it can be written as $\xi = (a, a+\tau_1+\tau_2, -a+\tau_1, -a+\tau_2)$ as well as $\xi=(a'',b'', -a''+\tau_3, -b''+\tau_3)$. Again, we may assume $a''= a$. Then we must have $-a+\tau_3= a+\tau_1+\tau_2$; equivalently, $a\in A_{2,\tau_1+\tau_2+\tau_3}$. Moreover, either we have $b''=-a+\tau_1$ or $b''=-a+\tau_2$. Note that, since $a\in A_{2, \tau_1+\tau_2+\tau_3}$, in the first case $-b''+\tau_3= -a+\tau_2$, while in the second case $-b''+\tau_3= - a+\tau_1$. 
	This shows that 
	$$W_{\tau_1}\cap W_{\tau_2}\cap W_{\tau_3} = \{ (a, a+\tau_1+\tau_2, -a+\tau_1, -a+\tau_2) \in A_0^{(4)}, \ a\in A_{2, \tau_1+\tau_2+\tau_3}\}. $$
	This set consists of $4$ distinct points. Indeed, $G$ acts transitively on it and the subgroup of $G$ generated by the $(\tau_i, -1)$, $i=1,2,3$ acts trivially. Moreover, it is easy to see that if $g\in G$ is not an element of this subgroup then it has no fixed points in the above set.
	Hence, the intersection $W_{\tau_1}\cap W_{\tau_2}\cap W_{\tau_3}$ is in bijection with $(\ZZ/2\ZZ)^2$. No point of this set belongs to the exceptional locus of the Hilbert-Chow morphism, so that $\overline{W_{\tau_1}}\cap \overline{W_{\tau_2}}\cap \overline{W_{\tau_3}}$ consists of $4$ distinct points.
	
	$(iii)$. Note that $W_{\tau_1}\cap W_{\tau_2}\cap W_{\tau_3}\cap W_{\tau_1+\tau_2+\tau_3}$ is empty. We have seen above that if $\theta\in A_2\setminus \{\tau_1, \tau_2, \tau_3, \tau_1+\tau_2+\tau_3\}$ the action of $(\theta, -1)$ on $W_{\tau_1}\cap W_{\tau_2}\cap W_{\tau_3}$ is free; as ${W_{\theta}}$ is clearly fixed by $(\theta, -1)$, it follows that $W_{\tau_1}\cap W_{\tau_2}\cap W_{\tau_3}\cap W_{\theta} =\emptyset$. 
\end{proof}
	
	\subsection{Fixed loci}
	We calculate the fixed locus of the automorphisms $g\in G$. Our results confirm those of Oguiso \cite{oguiso2020no} and Kamenova-Mongardi-Oblomkov \cite{KMO}. 
	
	\begin{proposition}\label{prop:FixedLoci}
		Let $g\neq 1\in G$. Then:
		\begin{itemize}
			\item if $g=(\tau, 1)$, with $\tau\neq 0\in A_2$, the fixed locus $K^g$ is the disjoint union of the $8$ $\mathrm{K}3$ surfaces $\overline{V_{\tau, \theta}}$, for $\theta$ varying in $A_2$; 
			\item if $g=(\tau, -1)$ with $\tau \in A_2$, the fixed locus $K^g$ is the union of the $\mathrm{K}3^{[2]}$-type variety $\overline{W_{\tau}}$ and $140$ isolated points. The isolated points are given by the set
			\[
			\ \ \ \{(\epsilon_1, \epsilon_2, \epsilon_3, \epsilon_4) , \ \epsilon_i \in A_{2,\tau} \text{ pairwise distinct and such that } \sum_{i=1}^4 \epsilon_i =0 \in A \}.
			\]
		\end{itemize}
	\end{proposition}
	\begin{proof} 
		Let $g=(\tau, 1)$ with $\tau\neq 0\in A_2$. The action of $g$ on $\xi=(a,b,c,d)\in A^{(4)}$ is given by $\xi\mapsto (a+\tau, b+\tau, c+\tau, d+\tau)$. We see immediately that the surfaces $V_{\tau, \theta}$ are fixed by $g$. Conversely, if $\xi $ is fixed and $a$ belongs to its support, then also $a+\tau$ must be in the support.
		Therefore, $\xi=(a, a+\tau, b, b+\tau)$; if $\xi\in A^{(4)}_0$, then $2a+2b=0$, so that $\theta\coloneqq a+b$ belongs to $A_2$ and $\xi\in V_{\tau, \theta}$. 
		By the description of $\overline{V_{\tau, \theta}}$ given in Lemma~\ref{lem:GeometricDescription}.$(i)$, the fixed locus of $g$ in $\nu^{-1}(V_{\tau, \theta})$ is precisely $\overline{V_{\tau, \theta}}$. 
		Since the fixed locus is smooth, $K^g$ is the disjoint union of the $8$ K$3$ surfaces $\overline{V_{\tau, \theta}}$, for $\theta\in A_2$. 
		
		Assume now that $g=(\tau, -1)$. If $\xi=(a,b,c,d) \in A^{(4)}$, the action of $g$ is given by $\xi\mapsto (-a+\tau, -b+\tau, -c+\tau, -d+\tau)$. Note that $W_{\tau}$ is contained in the fixed locus of $g$ on $A^{(4)}_0$. Moreover any $\xi$ entirely supported at points of $A_{2,\tau}$ is fixed by~$g$. 
		Conversely, let $\xi\in A_0^{(4)}$ be fixed. Then we have the following possibilities.
		If the support of $\xi$ does not intersect $A_{2, \tau}$, we have $\xi= (a, -a+\tau, b, -b+\tau)$ for some $a,b\in A\setminus A_{2, \tau}$ and $\xi $ belongs to $ W_{\tau}$. 
		If the support of $\xi$ contains some $\epsilon\in A_{2,\tau}$ and some $a\notin A_{2,\tau}$, then $\xi= (a, -a+\tau, \epsilon, b)$ for some $b\in A$; as $\xi\in A_0^{(4)}$, we must have $b=\epsilon$, and hence $\xi$ belongs to $ W_{\tau}$. 
		Finally, assume that $\xi=( \epsilon_1, \epsilon_2, \epsilon_3, \epsilon_4)$ is entirely supported at $A_{2,\tau}$. As $\xi$ belongs to $A^{(4)}_0$ we have $\sum_i \epsilon_i =0 $ in $A$.
		If the support of $\xi$ consists of $4$ distinct points, then $\xi\notin W_{\tau}$. Otherwise $\xi=(\epsilon_1, \epsilon_1, \epsilon_2, \epsilon_2)$ lies in $W_{\tau}$. 
		
		We conclude that the fixed locus of $g$ acting on $A^{(4)}_0$ consists of $W_{\tau}$ and the isolated points in the set 	
		\[
		\{(\epsilon_1, \epsilon_2, \epsilon_3, \epsilon_4) , \ \epsilon_i \in A_{2,\tau} \text{ pairwise distinct and such that } \sum_{i=1}^4 \epsilon_i =0 \in A \}.
		\]
		There are $140$ isolated fixed points. Indeed, the total number of ordered sequences $(\epsilon_1, \epsilon_2, \epsilon_3, \epsilon_4)$ such that $\sum_i \epsilon_i = 0 $ is $16^3$. There are $16$ such sequences with support a single point, and $6\binom{16}{2}=16 \cdot 45 $ with support two distinct points. The number of ordered sequences as above supported on four distinct point is $16^3-16-16\cdot 45=16\cdot 210$; under the symmetric group, each of them has an orbit of cardinality~$24$. Hence the number of isolated fixed points is $ \frac{210\cdot 16}{24} = 140$.
		
		Therefore, $K^g$ consists of $140$ isolated fixed points and the fixed locus of $g$ acting on $\nu^{-1}(W_\tau)$. To conclude we have to check that the latter coincides with $\overline{W_{\tau}}$. 
		Using the description of $\overline{W_{\tau}}$ given in the proof of Lemma \ref{lem:GeometricDescription}.$(ii)$ and the second assertion of Proposition~\ref{lem:local} it is readily seen that $\nu^{-1}(\xi) \cap K^g = \nu^{-1}(\xi) \cap \overline{W_{\tau}}$ for any  $\xi\in W_{\tau}$. 
		\end{proof}
		\begin{remark}\label{rmk:action}
			Let $h=(\epsilon, \pm 1)$ be an automorphism in $\Aut_0(K^3(A))=A_4\rtimes \langle -1\rangle$. One can see directly from Definition \ref{def:RelevantLoci} that $h$ maps $\overline{W_{\tau}} $ to $\overline{W_{\tau+2\epsilon}}$, for any $\tau\in A_2$.
		\end{remark}
		
	We can finally give the proof of the main result of this section. 
		
		\begin{proof}[Proof of Theorem \ref{thm:fixedLoci}]
			Let us first assume that $K$ is the generalized Kummer variety associated to an abelian surface $A$. 
			Let $g= (\tau, 1)\in G$, with $\tau\neq 0\in A_2$. Consider a component $\overline{V_{\tau, \theta}}$ of $K^g$. The surface $V_{\tau, \theta}$ is clearly contained in $W_{\theta}$, and hence the same holds for their strict transforms. 
			Let now $g= (\tau, -1)\in G$. Let $P$ be an isolated fixed point in $K^{g}$. Then we have $P=(\epsilon_1, \epsilon_2, \epsilon_3, \epsilon_4)$ for some pairwise distinct $\epsilon_i \in A_{2,\tau}$ summing up to $0$ in $A$. The last condition implies that we can write $P=(\epsilon_1, \epsilon_2, \epsilon_1+\theta, \epsilon_2+\theta)$ for some $0\neq \theta\in A_{2}$. This is the same as 
			$$ P = (\epsilon_1, \epsilon_2, -\epsilon_1+\theta+\tau, -\epsilon_2+\theta+\tau),$$
			which lies in $\overline{W_{\tau+\theta}}$. By Lemma \ref{prop:FixedLoci} we conclude that for a generalized Kummer variety the union of the fixed loci $K^g$ consists of the $16$ varieties $\overline{W_{\tau}}$ of $\mathrm{K}3^{[2]}$-type. The intersections of these components were computed in Lemma \ref{lem:intersections}.
			
			Let now $f\colon \mathcal{K}\to B$ be a smooth and proper family of $\mathrm{Kum}^3$-manifolds over a connected base $B$, with a fibre $\mathcal{K}_0$ isomorphic to the generalized Kummer variety on an abelian surface $A$.
			By the result of Hassett-Tschinkel \cite[Theorem 2.1]{hassettTschinkel}, up to a finite \'etale base-change, there is a fibrewise action of $G=A_2\times\langle -1 \rangle$ on $\mathcal{K}$. By \cite[Lemma 3.10]{fujiki1983}, this action is locally trivial with resepect to $f$, i.\ e.\ any $x\in \mathcal{K}$ has a neighborhood $U_x=(f^{-1}f(U_x)\cap \mathcal{K}_{f(x)} )\times f(U_x)$ such that the action on $U_x$ of the stabilizer subgroup of $x$ is induced by that on the fibre $\mathcal{K}_{f(x)}$. 
			Therefore, for any $\tau\in A_{2}$ the fixed locus $\mathcal{K}^g$ of $g=(\tau, -1)$ contains a smooth family $\overline{\mathcal{W}_{\tau}}\to B$ of manifolds of $\mathrm{K}3^{[2]}$-type as the unique component of~$\mathcal{K}^g$ with positive dimensional intersection with the fibres of $f$. 
			We conclude that for any~$b\in B$ we have $$\bigcup_{1\neq g\in G} (\mathcal{K}_b)^g = \bigcup_{\tau\in A_2} \overline{\mathcal{W}_{\tau}}_{\lvert_{\mathcal{K}_b}},$$
			and that the components $\overline{\mathcal{W}_{\tau}}_{\lvert_{\mathcal{K}_b}}$ intersect as claimed.
			\end{proof}

	\section{A Lagrangian fibration}\label{sec:specialCase}
	
	 Throughout this section, $J$ denotes a general principally polarized abelian surface.
	 We fix a symmetric theta divisor $\Theta\subset J$, which is unique up to translation by a point of $J_2$. For any integer $d$, we consider the Mukai vector $v_d\coloneqq (0, 2\Theta, d-4)$ on~$J$ and the moduli space $M_J(v_d)$ of $\Theta$-semistable sheaves with Mukai vector $v_d$ (\cite{huybrechts2010geometry}).
	 The Albanese fibration
	 \[
	 M_J(v_d) \xrightarrow{ \ \mathrm{alb} \ } \Pic^0(J)\times J,
	 \]
	 is isotrivial \cite{Yos01}. We denote the fibre of $\mathrm{alb}$ over $(\mathcal{O}_J, 0 )$ by $K_J(v_d)$. For $d$ even, $K_J(v_d)$ is singular and admits a crepant resolution of $\mathrm{OG}6$-type, by \cite{O'G03} and \cite{LS06}.
	 When $d$ is odd, $K_J(v_d)$ is smooth, and it is a hyper-K\"{a}hler sixfold of $\mathrm{Kum}^3$-type \cite{Yos01}. In this case, we consider the group $G \subset \Aut_0(K_{J}(v_d))$ introduced in Definition \ref{def:G}. 
	
	The main result of this section is the following. 
	
	\begin{theorem}\label{thm:main}
		The quotient $K_J(v_3)/ G$ is birational to $\Km(J)^{[3]}$.
	\end{theorem}

We will in fact show that for any odd $d$ the quotient $K_J(v_d)/G$ is birational to a variety of $\mathrm{K}3^{[3]}$-type. Before proceeding with the proof we need to fix some notation.

\subsection{Preliminaries}\label{subsec:mapsdoublecover} 
Let $C$ be a smooth projective curve of genus $g$. Its Picard variety $\Pic^{\bullet}(C)=\bigsqcup_d \Pic^d(C)$ is the group of line bundles modulo isomorphism. Tensor product gives multiplication maps $m\colon \Pic^{a}(C) \times \Pic^{b}(C) \to \Pic^{a+b}(C)$. We denote by $[n]\colon \Pic^d(C)\to \Pic^{nd}(C)$ the $n$-th power map.
The Jacobian of $C$ is the abelian variety $\Pic^0(C)$; for $d\neq 0$, the variety $\Pic^d(C)$ is a torsor under $\Pic^0(C)$. 

Assume now that $\pi\colon \tilde{C}\to C$ is an \'etale double cover, so that $\tilde{C}$ is of genus $2g-1$. In this situation, for each integer $d$, we have:
\begin{itemize}
	\item the pull-back map $\pi^*\colon \Pic^{d}(C)\to \Pic^{2d}(\tilde{C})$;
	\item the norm map $\mathrm{Nm}_\pi\colon\Pic^d(\tilde{C})\to \Pic^d({C})$;
	\item the covering involution $\sigma\colon \tilde{C}\to \tilde{C}$, inducing $\sigma^*\colon \Pic^{d}(\tilde{C})\to \Pic^d(\tilde{C})$.
\end{itemize}
The composition $\mathrm{Nm}_\pi \circ\pi^*\colon \Pic^d(C)\to \Pic^{2d}(C)$ coincides with multiplication by $2$, while $\pi^*\circ \mathrm{Nm}_\pi\colon \Pic^d(\tilde{C})\to \Pic^{2d}(\tilde{C})$ is given by $L\mapsto L\otimes \sigma^*(L)$.

We refer to Mumford's paper \cite{mumford1974} for the following.
The \'etale double cover $\pi$ is uniquely determined by a $2$-torsion line bundle $\eta$ on $C$, which is the only non-trivial element in $\ker (\pi^*\colon \Pic^{0}(C)\to \Pic^0(\tilde{C}))$.
Moreover, the image of the pull-back map~$\pi^*\colon \Pic^{d}(C)\to \Pic^{2d}(\tilde{C})$ is precisely the fixed locus of~$\sigma^*$. 
The involution $\sigma^*$ acts as the inverse $L\mapsto L^{\vee}$ on the kernel of the norm map $\mathrm{Nm}_\pi\colon \Pic^0(\tilde{C})\to \Pic^0(C)$. 
This kernel has two connected components; the one containing the neutral element is the Prym variety $ P(\pi)$ of the cover, an abelian variety of dimension $g-1$.
The fibres of $\mathrm{Nm}_\pi\colon \Pic^d(\tilde{C})\to \Pic^d(C)$ are torsors under $P(\pi)\times\ZZ/2\ZZ$.

\subsection{Geometric set-up} \label{subsec:setup}
Let $\Theta \subset J$ be the general principally polarized abelian surface of Picard rank~$1$, with $\Theta$ a symmetric theta divisor. By Riemann-Roch $H^0(J, \mathcal{O}_J(2\Theta))=4$; we will identify the complete linear system $\lvert 2\Theta\rvert$ with $\PP^3$.
It is classically known that this linear system is base-point free and induces an embedding of $J/\pm 1$ into $\PP^{3,\vee}$ as a quartic surface with $16$ nodes. Blowing-up the nodes of this quartic surface one obtains the Kummer K3 surface $\Km(J)$ associated to $J$. We denote by~$H$ the divisor on $\Km(J)$ obtained as pull-back of a hyperplane section of $J/\pm 1\subset \PP^{3,\vee}$. 
The linear system $|H|$ is naturally identified with $\lvert 2\Theta\rvert=\PP^3$. 

Consider the universal family of genus $5$ curves $\tilde{\mathcal{C}}\to \lvert 2\Theta\rvert$. The involution $-1$ of~$J$ acts trivially on $\lvert 2\Theta\rvert$, and we consider the quotient family $\mathcal{C}\to \PP^3$, which is identified with the hyperplane linear system on $J/\pm 1\subset \PP^{3,\vee}$. There is a degree $2$ morphism
$
\pi\colon \tilde{\mathcal{C}}\to \mathcal{C}
$
over $\PP^3$. We denote by $\mathcal{D}\to \lvert H\rvert $ the universal family of curves over the linear system~$\lvert H\rvert $ on~$\Km(J)$; clearly, $\mathcal{C}_b=\mathcal{D}_b$ if and only if $\mathcal{C}_b$ does not pass through any of the nodes of $J/\pm 1$.

	A detailed study of the linear system $\lvert 2\Theta\rvert$ and the map $\pi\colon \tilde{\mathcal{C}}\to \mathcal{C}$ can be found in Verra's article \cite{verra}. In particular, he shows that, for any $b\in \PP^3$, the curve $\tilde{\mathcal{C}}_b$ is smooth if and only if $\mathcal{C}_b$ is so, and in this case the Prym variety of the \'{e}tale double cover $\pi_b \colon \tilde{\mathcal{C}}_b \to \mathcal{C}_b$ is the abelian surface $J$. 
	Moreover $\mathcal{C}_b$ is smooth if and only if it does not contain any node of $J/\pm 1$ and it does not lie on a tangent hyperplane.

\subsection{Beauville-Mukai systems}
Taking the Jacobian of the smooth curves in the families introduced above yields families of abelian varieties over a Zariski open subvariety of $\PP^3$. The total space of these families can be compactified considering suitable moduli spaces of stable sheaves on $J$ and $\mathrm{Km}(J)$. The construction leads to certain hyper-K\"{a}hler varieties equipped with Lagrangians fibrations, called Beauville-Mukai systems~\cite{beauville1991}, \cite{Muk84}.

Consider the Mukai vector $v_d=(0, 2\Theta, d-4)$ on $J$, for an integer $d$. The moduli space $M_J(v_d)$ parametrizes pure dimension~$1$ sheaves on $J$ which are push-forward of semistable and torsion free sheaves of rank~$1$ and degree $d$ supported on curves algebraically equivalent to $2\Theta$. Mapping a sheaf to its support \cite{lepotier} gives a morphism
\[
M_J(v_d) \xrightarrow{\ \mathrm{supp} \ } \Pic^0(J)\times \PP^3,
\]
We define the degree $d$ relative compactified Jacobian as
\[
\overline{\Pic}^d(\tilde{\mathcal{C}}/ \PP^3) \coloneqq \mathrm{supp}^{-1}(\{ \mathcal{O}_J \} \times \PP^3).
\]
If $b\in \PP^3$ corresponds to a smooth curve in $|2\Theta|$, the fibre over $b$ of the support morphism $\mathrm{supp}\colon \overline{\Pic}^d(\tilde{\mathcal{C}}/\PP^3) \to \PP^3$ is the degree~$d$ Picard variety of the curve.
The Albanese morphism $\mathrm{alb}\colon \overline{\Pic}^d(\tilde{\mathcal{C}}/\PP^3) \to J$ is an isotrivial fibration with fibre $K_J(v_d)$, and it maps a sheaf $F\in K_J(v_d)$ to the sum $\sum c_2(F)\in J$, see \cite[\S6]{Wie18}. 

A similar construction can be done for the family of curves $\mathcal{D}$ on $\Km(J)$.
Given an integer $d$, we consider the Mukai vector $w_d= (0, H, d-2)$ on $\Km(J)$. The degree~$d$ relative compactified Jacobian of curves in $|H|$ is the moduli space
\[
\overline{\Pic}^d(\mathcal{D}/ \PP^3) \coloneqq M_{\Km(J)}(w_d),
\]
of semistable sheaves with Mukai vector $w_d$, with respect to a fixed $w_d$-generic polarization. It is a hyper-K\"{a}hler variety of $\mathrm{K}3^{[3]}$-type by \cite{O'G97}, \cite{yoshioka99}. The moduli space $M_{\Km(J)}(w_d)$ parametrizes pure dimension $1$ sheaves on $\Km(J)$ which are pushforward of torsion free sheaves of rank $1$ and degree $d$ on curves in the linear system~$\lvert H \rvert$. There is a morphism $\mathrm{supp} \colon \overline{\Pic}^d(\mathcal{D}/ \PP^3)\to \PP^3$ which is a Lagrangian fibration, whose general fibres are the degree $d$ Picard varieties of the smooth curves in $\lvert H \rvert$.

\subsection{The relative norm map} 
We denote by $B\subset \PP^3$ the locus parametrizing smooth curves. The maps introduced in \S\ref{subsec:mapsdoublecover} give morphisms over $B$. Hence, we obtain rational maps
\begin{alignat*}{2}
\mathrm{Nm}_\pi & \colon \overline{\Pic}^d(\tilde{\mathcal{C}}/\PP^3)  && \dashrightarrow \overline{\Pic}^{d}({\mathcal{D}}/\PP^3);\\
\pi^* & \colon \overline{\Pic}^d(\mathcal{D}/\PP^3) && \dashrightarrow  \overline{\Pic}^{2d}(\tilde{\mathcal{C}}/\PP^3).
\end{alignat*}
The pull-back $\pi^*$ has been studied by Rapagnetta \cite{rapagnetta2007topological} and by Mongardi-Rapagnetta-Sacc\`a~\cite{MRS18}. They show that it gives a degree $2$ dominant rational map
\[
\overline{\Pic}^d(\mathcal{D}/\PP^3) \dashrightarrow K_J(v_{2d}),
\]
exhibiting a variety of OG6-type as quotient of a variety of $\mathrm{K}3^{[3]}$-type by a birational symplectic involution. 
We consider instead the norm map $\mathrm{Nm}_{\pi}$.

\begin{lemma}\label{lem:restriction}
	For each $d$, the restriction $\mathrm{Nm}_\pi\colon K_J(v_d) \dashrightarrow \overline{\Pic}^{d}({\mathcal{D}}/\PP^3)$ is dominant and generically finite of degree $2^5$.
\end{lemma}
\begin{proof}
	For each $k$, the $k$-th power map $[k]\colon \overline{\Pic}^d(\tilde{\mathcal{C}}/\PP^3) \dashrightarrow \overline{\Pic}^{kd}(\tilde{\mathcal{C}}/\PP^3)$ restricts to a rational map $[k]\colon  K_J(v_d) \dashrightarrow K_J(v_{kd})$. This can be easily checked at smooth curves $\tilde{\mathcal{C}_b}$ in the family using the description of the Albanese map $L\mapsto\sum c_2(L)$.

	Consider the commutative diagram
	\[
	\begin{tikzcd}
		K_J(v_d) \arrow[dashed]{r}{[2]} \arrow[dashed]{d}{\mathrm{Nm}_\pi}  & K_J(v_{2d})  \arrow[dashed]{d}{\mathrm{Nm}_\pi} \\
		\overline{\Pic}^{d}({\mathcal{D}}/\PP^3) \arrow[dashed]{r}{[2]} & \overline{\Pic}^{2d}({\mathcal{D}}/\PP^3).
	\end{tikzcd}
	\]
	Both the horizontal maps are dominant and generically finite of the same degree~$2^6$. Indeed, they preserve the fibres of the respective support morphisms, which, generically, are torsors under abelian threefolds. Hence, the degree of $[2]$ is the number of points of order $2$ on an abelian threefold. 
	
	It thus suffices to prove the lemma for $\mathrm{Nm}_\pi \colon K_J(v_{2d}) \dashrightarrow \overline{\Pic}^{2d}({\mathcal{D}}/\PP^3)$. Since $K_J(v_{2d})$ is the closure of the image of $\pi^*$, the involution $\sigma^*$ is the identity on $K_J(v_{2d})$. Hence, the rational map
	\[
	\begin{tikzcd}
		K_J(v_{2d}) \arrow[dashed]{r}{\mathrm{Nm}_\pi\, } 
		&
		\overline{\Pic}^{2d}({\mathcal{D}}/\PP^3) \arrow[dashed]{r}{\pi^* } & K_J(v_{4d}) 
	\end{tikzcd} 
	\]
	coincides with multiplication by $[2]$; this composition is therefore dominant and generically finite of degree $2^6$. Since $\pi^*\colon \overline{\Pic}^{2d}({\mathcal{D}}/\PP^3)\dashrightarrow K_J(v_{4d})$ is generically finite of degree $2$, it follows that $\mathrm{Nm}_\pi$ is dominant and it has degree $2^5$.
	\end{proof}

	\subsection{The action of $G$}
	Assume now that $d$ is odd, so that $K_J(v_d)$ is a smooth variety of $\mathrm{Kum}^3$-type and the support map gives a Lagrangian fibration on it. In this case the group $\Aut_0(K_J(v_d))$ has been explicitly identified by Kim in \cite{Kim21}.
	
	\begin{proposition} \label{prop:identifyG}
		Let $d$ be odd. Then $G=\Pic^0(J)_2 \times \langle -1\rangle $, where $-1$ acts on $K_J(v_d)$ via the pull-back $F\mapsto (-1)^*(F)$ of sheaves, and the action of $L\in \Pic^0(J)_2$ is given by $F\mapsto F\otimes L$. Any element of $G$ preserves the fibres of the support fibration.		
	\end{proposition}
	\begin{proof}
		We first show that $\Pic^0(J)_2\times\langle -1\rangle$ acts on $K_J(v_d)$ via automorphisms trivial on the second cohomology and which preserve the Lagrangian fibration given by the support morphism.
		
		Since $-1\colon J\to J$ acts trivially on the cohomology of $J$ in even degrees, the pull-back of sheaves defines an automorphism $(-1)^*$ of $K_J(v_d)$; this automorphism is trivial for $d$ even but not for $d$ odd. By \cite[Lemma~2.34]{mongardi2015}, the action of $(-1)^*$ is symplectic, i.~e., the induced action on~the transcendental cohomology $H^2_{\mathrm{tr}}(K_J(v_d), \ZZ)$ is the identity. Since all curves in $\lvert 2\Theta\rvert $ are stable under $-1$, the automorphism $(-1)^*$ preserves all fibres of the support fibration. 
		The Picard rank of $K_J(v_d)$ is $2$, because~$J$ is a general abelian surface by assumption (see \cite{Yos01}). Hence, by Lemma~\ref{lem:observationLagrangian} below, the induced action of $(-1)^*$ on $\NS(K_J(v_d))$ is also the identity. Since $H^2(K_J(v_d), \ZZ)$ is torsion-free, we conclude that $(-1)^*$ acts trivially on the second cohomology. 
		
		Let now $L\in \Pic^0(J)_2$.
		Let $B\subset \lvert 2\Theta\rvert $ be the locus of smooth curves, and consider the universal curve $j\colon \tilde{\mathcal{C}}_B\hookrightarrow J\times B$. By \cite[Lemma 6.9]{Wie18}, the corresponding pull-back $j^*\colon \Pic^0(J)\times B \to \overline{\Pic}^0(\tilde{\mathcal{C}}/\PP^3)_B$ is an injection of abelian schemes. Note that the covering involution $\sigma$ for the cover $\pi\colon \tilde{\mathcal{C}}_B\to {\mathcal{C}}_B$ is the restriction of $-1\times \mathrm{id}_B$. Hence, the induced involution $\sigma^*$ on $\overline{\Pic}^0(\tilde{\mathcal{C}}/\PP^3)_B$ is the inverse map on $j^*_b(\Pic^0(J)) $, for each~$b\in B$. As~$K_J(v_0)$ is the fixed locus of this involution, the $2$-torsion line bundle $L$ on~$J$ corresponds uniquely to a $2$-torsion section $s_L$ of $K_J(v_0)_B\to B$. For any $d$, the variety $K_J(v_d)_B$ is a torsor under $K_J(v_0)_B$ and hence $F\mapsto F\otimes L$ induces a birational automorphism $g_L$ of $K_J(v_d)$, which restricts to a translation on the smooth fibres of the support morphism. This implies that $g_L$ is symplectic. We deduce from Lemma \ref{lem:observationLagrangian} that $g_L$ extends to a regular automorphism of $K_J(v_d)$, whose induced action is trivial on $\NS(K_J(v_d))$ as well.
		
		According to \cite[Theorem 5.1]{Kim21} we have described all automorphisms in $\Aut_0(K_J(v_d))$ preserving the fibration $\mathrm{supp}\colon K_J(v_d)\to \PP^3$. We claim now that an automorphism in $\Aut_0(K_J(v_d))$ which does not preserve the fibration cannot fix a component of dimension $4$ on $K_J(v_d)$; this will show that $G=\Pic^0(J)_2\times \langle -1\rangle$. 
		
		Let $D = \mathrm{supp}^* (\mathcal{O}_{\PP^n}(1)) \in \Pic(K_J(v_d))$ be the line bundle inducing the fibration. Since $\Aut_0(K_J(v_d))$ acts trivially on $H^2(K_J(v_d),\ZZ)$, it acts on $|D|= \PP^3=|2\Theta|$. We obtain an action of $\Aut_0(K_J(v_d))/(\Pic^0(J)_2\times \langle -1\rangle) \cong (\ZZ/2\ZZ)^4$ on $|2\Theta|$, which, up to conjugation by an automorphism of $\PP^3$, is identified with the action generated by
		\begin{align*}
		(x,y,z,w) & \mapsto (z,w,x,y),\ \ \ & (x,y,z,w) & \mapsto (y,x,w,z), \\
		 (x,y,z,w) & \mapsto (x,y,-z,-w), \ \ \ & (x,y,z,w) & \mapsto (x, -y, z, -w);
		\end{align*}
	 see \cite[Lemma 1.52, Note 1.4]{del1994}. Thus, $h\in \Aut_0(K_J(v_d))$ either acts trivially on $|2\Theta|$ or it fixes a pair of skew lines. Assume by contradiction that $h\in \Aut_0(K_J(v_d))$ fixes a fourfold but does not preserve the support fibration. Then $(K_J(v_d))^h$ contains one of the varieties ${Z_{j}}$ of $\mathrm{K}3^{[2]}$-type from Theorem~\ref{thm:fixedLoci}, and the image of $Z_j$ via the fibration must be a line $R \subset |2\Theta|$ fixed by~$h$. But, by a theorem of Matsushita \cite{matsushita}, the $\mathrm{K}3^{[2]}$-variety $Z_j$ does not admit any non-constant morphism to $\PP^1$.
	\end{proof}

	\begin{remark}\label{rmk:relativePrym}
	By \cite{verra}, the relative Prym variety $\mathcal{P}\to B$ of the double cover $\pi\colon \tilde{\mathcal{C}}_B\to \mathcal{C}_B$ is an isotrivial family over $B$ with fibre $J$. Since the involution $\sigma^*$ acts as the inverse on the image of the pull-back $j^*$, the constant abelian scheme $j^*(\Pic^0(J)\times B)$ is contained into, and hence equal to, the relative Prym variety.	
	\end{remark} 

	\begin{lemma}\label{lem:observationLagrangian}
		Let $f\colon X\to \PP^n$ be a Lagrangian fibration, where $X$ is a $2n$-dimensional projective hyper-K\"ahler variety of Picard rank $2$. Let $g\colon X\dashrightarrow X$ be a birational automorphism such that $f\circ g = f$. Then $g$ extends to a regular automorphism of $X$, and $g^*_{\lvert_{ \NS(X)}}\colon \NS(X)\to \NS(X)$ is the identity.
	\end{lemma} 
	\begin{proof}
		The N\'eron-Severi group $\NS(X)$ is a rank $2$ lattice of signature $(1,1)$. Any birational automorphism $g$ induces an isometry $g^*$ of $H^2(X,\ZZ)$ which restricts to an isometry of $\NS(X)$. 
		Denote by $D\in \NS(X)$ the class of the line bundle $f^*(\mathcal{O}_{\PP^n}(1))$, and pick $E\in \NS(X)$ such that $D$ and $E$ generate $\NS(X)\otimes \QQ$. Since $f\circ g=f$, we have $g^*(D)=D$. Using that $D$ is isotropic one sees that this this forces $g^*(E)=E$ as well, so that $g^*$ is the identity on $\NS(X)$; by \cite{fujiki1981}, $g$ extends to an automorphism~of~$X$.
	\end{proof}

We can now complete the proof of the main result of this section. 

\begin{proof}[Proof of Theorem \ref{thm:main}]
	Using the idenfication of $G$ given in Proposition \ref{prop:identifyG} we will show that the norm map descends to a birational map $K_J(v_d)/G\dashrightarrow \overline{\Pic}^d(\mathcal{D}/\PP^3)$, for any odd $d$. The varieties $ \overline{\Pic}^d(\mathcal{D}/\PP^3)$ are of $\mathrm{K}3^{[3]}$-type, and $\overline{\Pic}^3(\mathcal{D}/\PP^3)$ is birational to $\Km(J)^{[3]}$, see \cite[Proposition 1.3]{beauville1999}.
	
	For a smooth curve $\tilde{\mathcal{C}}_b \in \lvert 2\Theta \rvert $ the norm map is induced by the map of divisors which sends $\sum_i a_i[P_i] $ to $\sum_i a_i [\pi(P_i)]$. It is then clear that $\mathrm{Nm}_\pi((-1)^*(F))= \mathrm{Nm}_\pi(F)$ for any~$F\in K_J(v_d)_B$. Let instead $L\in \Pic^0(J)_2$, and let $s_L\colon \PP^3 \dashrightarrow K_J(v_0) $ be the rational section defined by $s_L(b)= j_b^*(L)$, where $j\colon \tilde{\mathcal{C}}_B\hookrightarrow J\times B$ denotes the natural inclusion. By Remark \ref{rmk:relativePrym}, the section $s_L$ is contained in the relative Prym variety; in particular, the composition $\mathrm{Nm}_\pi\circ s_L$ gives the zero section of $\overline{\Pic}^0(\mathcal{D}/\PP^3)_B\to \PP^3$. 
		
	Therefore, for any $g\in G$ and $y\in K_J(v_d)_B$ we have $\mathrm{Nm}_\pi(g(y))= \mathrm{Nm}_\pi(y)$, which means that $\mathrm{Nm}_\pi\colon K_J(v_d)\dashrightarrow \overline{\Pic}^d(\mathcal{D}/\PP^3)$ descends to a rational map 
	\[
	\overline{\mathrm{Nm}}_\pi\colon K_J(v_d)/G \dashrightarrow \overline{\Pic}^d(\mathcal{D}/\PP^3).
	\]
	This map is in fact birational, because $\mathrm{Nm}_\pi\colon K_J(v_d)\dashrightarrow \overline{\Pic}^d(\mathcal{D}/\PP^3)$ is generically finite of degree $2^5$ by Lemma \ref{lem:restriction}, and $G$ has also order $2^5$. 
\end{proof}

\section{The K3 surface associated to a $\mathrm{Kum}^3$-variety}\label{sec:AssociatedK3}

In this section we give the proof of Theorems \ref{thm:MainResult} and \ref{thm:AssociatedK3}. We let $K$ be a manifold of $\mathrm{Kum}^3$-type and consider the quotient $K/G$ by the group $G\cong (\ZZ/2\ZZ)^5$ of Definition~\ref{def:G}. We will show that the blow-up of the singular locus of $K/G$ yields a hyper-K\"ahler manifold~$Y_K$ of $\mathrm{K}3^{[3]}$-type.

\subsection{The singularities of $K/G$}
By Theorem \ref{thm:fixedLoci}, the locus $Z=\bigcup_{1\neq g\in G} K^g$ is the union of $16$ irreducible components $Z_i$, for~$i=1,\dots, 16$. Denote by $X_i  \subset K/G$, the image of $Z_i$ via the quotient map $q\colon K\to K/G$. We introduce the following stratification of $K/G$ into closed subspaces:
\[
X^j \coloneqq  \{x \in K/G \ \vert \ x \text{ belongs to at least } j \text{ components } X_i\};
\]
clearly, $X^0=K/G$.

\begin{proposition}\label{prop:singularities}
The subspace $X^j$ is empty for $j\geq 4$. For $j<4$, a point $x\in X^j\setminus X^{j+1}$ has a neighborhood $U_x\subset K/G$ analytically isomorphic to 
	\[
	(\CC^2/\iota)^{j} \times (\CC^2)^{3-j},
	\]
	where $\iota$ is the involution $(x,y)\mapsto (-x,-y)$.
\end{proposition}
\begin{proof}
	Theorem \ref{thm:fixedLoci} implies immediately that $X^j$ is empty for $j\geq 4$. The group $G$ acts freely on $q^{-1}(X^0\setminus X^1)$, and hence $X^0\setminus X^1$ is smooth.	
	
	If $\mathcal{K}\to B$ is a smooth proper family of $\mathrm{Kum}^3$-manifolds, the quotient $\mathcal{K}/G \to B$ is a locally trivial family (see \cite[Lemma 3.10]{fujiki1983}): any point $x\in \mathcal{K}/G$ has a neighborhood $U_x$ of the form to $f(U_x)\times (f^{-1}(f(x))\cap U_x)$. Therefore, to prove the proposition we may assume that $K$ is the generalized Kummer sixfold on an abelian surface $A$.
	
	In this case the components $Z_{i}$ are the explicit $\overline{W_{\tau}}$, for $\tau \in A_2$ (see Definition~\ref{def:RelevantLoci}). Recall that $\overline{W_{\tau}}\subset K$ is the unique positive dimensional component of the fixed locus of $(\tau, -1)\in G$, and that the induced action of $G/\langle (\tau,-1) \rangle$ on $\overline{W_{\tau}}$ is faithful. 
	For any $z\in \overline{W_{\tau}}$ there is a decomposition $T_z(K)=N_{\overline{W_{\tau}}\vert K, z}\oplus T_z(\overline{W_{\tau}})$, where the first factor is the normal space. The action of $(\tau, -1)$ on $T_z(K)$ is $(-1, 1)$. 
	Let now $x\in X^1 \setminus X^2$, and let $z\in K$ be a preimage of $x$. By definition, $z$ belongs to exactly one component $\overline{W_{\tau}}$. The stabilizer is $G_z=\langle (\tau,-1) \rangle$. By the above, there exists a neighborhood $V_z$ of $z\in K$ such that $g(V_z)\cap V_z$ is empty for any $g \notin G_z$, and $V_z=(\CC^2)^3$ with $(\tau, -1)_{\vert_{V_z}} =(\iota, \mathrm{id}, \mathrm{id})$. The image of $V_z$ under the quotient map is thus a neighborhood $U_x=(\CC^2/\iota)\times (\CC^2)^2$ of $x$ in $K/G$.
	
	In the other cases we proceed similarly. Let $x\in X^2\setminus X^3$, and let $z\in K$ be one of its preimages. Then $z$ belongs to two distinct components $\overline{W_{\tau_1}}$ and $\overline{W_{\tau_2}}$. The stabilizer $G_z\cong (\ZZ/2\ZZ)^2$ is the subgroup $\langle(\tau_1,-1), (\tau_2,-1)\rangle$ of $G$. There is a decomposition $T_z(K) = N_{\overline{W_{\tau_1}}\vert K, z} \oplus N_{\overline{W_{\tau_2}}\vert K, z} \oplus T_z(\overline{W_{\tau_1}}\cap \overline{W_{\tau_2}})$, and the action of $G_z$ is generated by $(-1, 1, 1)$ and $(1, -1, 1)$. This implies that the image in $K/G$ of a sufficiently small neighborhood of $z$ in $K$ is isomorphic to $(\CC^2/\iota)^2\times (\CC^2)$.
	
	Finally, let $x\in X^3$ and $z\in K$ a preimage of it. Then $z$ belongs to exactly $3$ components $\overline{W_{\tau_i}}$, for $i=1,2,3$. The stabilizer $G_z$ is the subgroup generated by the involutions $(\tau_i, -1)$, for $i=1,2,3$, and there is a decomposition $T_z(K)=\bigoplus_{i=1}^3 N_{\overline{W_{\tau_i}}\vert K,z}$ of the tangent space. The action of $(\tau_i,-1)$ on $T_{z}(K)$ is $-1$ on the $i$-th summand and the identity on the complement. We conclude that there exists a neighborhood $V_z$ of $z$ in~$K$ whose image in $K/G$ is isomorphic to $(\CC^2/\iota)^3$.
\end{proof}

\subsection{The symplectic resolution of $K/G$} 
We will now conclude the proof of Theorem \ref{thm:MainResult} in several steps. First we show that the blow-up of the singular locus of~$K/G$ resolves the singularities.

\begin{proposition}\label{prop:resolution}
	Let $Y_K\coloneqq \mathrm{Bl}_{X^1}(K/G)$ be the blow-up of the singular locus with reduced structure. Then $Y_K$ is a smooth manifold. It is identified with the quotient by $G$ of the blow-up of $K$ at $Z=\bigcup_{1\neq g\in G} K^g$.
\end{proposition}

The proposition will be reduced to the following elementary statement.

\begin{lemma}\label{lem:blowuplocal}
	Let $\iota \colon \CC^2 \to \CC^2$ be the involution $(x,y)\mapsto (-x, -y)$, let $j,k$ be positive integers.
	Then the blow-up of $(\CC^2/\iota)^j\times (\CC^2)^k$ along its singular locus is isomorphic to $( \mathrm{Bl}_0 (\CC^2/\iota) )^j\times (\CC^2)^k$. It is thus smooth, and identified with the quotient $(\mathrm{Bl}_0(\CC^2)/\iota)^j\times (\CC^2)^k$ of $(\mathrm{Bl}_0(\CC^2))^j\times (\CC^2)^k$. 
\end{lemma}	
\begin{proof}
	Given noetherian schemes $T_1$ and $T_2$ and closed subschemes $V_1\subsetneq T_1$ and $V_2\subsetneq T_2$, we have
	\[
	\mathrm{Bl}_{(V_1\times T_2) \cup ( T_1\times V_2)} (T_1\times T_2) = \mathrm{Bl}_{V_1}(T_1) \times \mathrm{Bl}_{V_2}(T_2).
	\]
	It is easy to see this when $T_1$ and $T_2$ are affine schemes, using the definition of blowing-up via the Proj construction (\cite[II, \S7]{Hartshorne}). 
	The singular locus of $(\CC^2/\iota)^j\times (\CC^2)^k$ is the union of $ \pr_i^{-1}(\{0\})$ for $i$ from $1$ to $j$, where $\pr_i\colon (\CC^2/\iota)^j\times (\CC^2)^k \to (\CC^2/\iota)$ is the projection onto the $i$-th factor.
	Via the above observation, the statement is reduced to the case $j=1$, $k=0$, which is the well-known minimal resolution of an isolated $A_1$-singularity of a surface.
\end{proof} 

\begin{proof}[Proof of Proposition \ref{prop:resolution}]
	By Proposition \ref{prop:singularities}, any point $x\in K/G$ has a neighborhood of the form $(\CC^2/\iota)^j\times (\CC^2)^k$. Therefore, Lemma \ref{lem:blowuplocal} immediately implies that the blow-up $Y_K$ of the singular locus is smooth. It also implies that the natural birational map $\mathrm{Bl}_Z(K)/G\dashrightarrow Y_K$ extends to an isomorphism.
\end{proof}

We will now use a criterion due to Fujiki \cite{fujiki1983} to show that $Y_K$ is symplectic, i.\ e.\ that it admits a nowhere degenerate holomorphic $2$-form.
\begin{proposition}\label{prop:symplectic}
	The manifold $Y_K$ is hyper-K\"{a}hler. 
\end{proposition} 
\begin{proof}
	The singular space $K/G$ is a primitively symplectic V-manifold in the sense of Fujiki. According to his \cite[Proposition 2.9]{fujiki1983}, to show that $Y_K$ admits a symplectic form it suffices to check that the following two conditions are satisfied:
	\begin{enumerate}[label=(\roman*)]
		\item for each component $X_j$ of the singular locus of $K/G$, a general point $x\in X_j$ has a neighborhood $U_x = A \times V_x$ in $K/G$, where $A$ is a surface and $V_x=U_x\cap X_j$ is a smooth neighborhood of $x$ in $X_j$;
		\item the restriction of $p\colon Y_K\to K/G$ to the preimage of $U_x$ is the product $$p'\times \mathrm{id}\colon \tilde{A}\times V_x \to A\times V_x, $$ where $p'$ is the minimal resolution of $A$.	
	\end{enumerate} 
	We have already shown in the course of proof of Propositions \ref{prop:singularities} and \ref{prop:resolution} that these conditions hold, with $A$ a neighborhood of the ordinary node $0\in \CC^2/\iota$, whose minimal resolution is the blow-up of the singular point.
	
	Hence, $Y_K$ is symplectic. From its description as the quotient $\mathrm{Bl}_Z(K)/G$ we obtain $h^{2,0}(Y_K)=1$ (see also \S\ref{subsec:associateK3}). Finally, $Y_K$ is simply connected by \cite[Lemma 1.2]{fujiki1983}.
\end{proof}

We can now complete the proof of our main result.

\begin{proof}[Proof of Theorem \ref{thm:MainResult}]
 	Let $\mathcal{K}\to B$ be a smooth proper family of manifolds of $\mathrm{Kum}^3$-type. Up to a finite \'etale base-change, $G$ acts fibrewise on $\mathcal{K}$, and the fixed loci of automorphisms in $G$ are smooth families over $B$. By Theorem \ref{thm:fixedLoci}, the union $\mathcal{Z}$ of $\mathcal{K}^g$ for~$1\neq g\in G$ has $16$ components $\mathcal{Z}_i$, each of which is a smooth family of manifolds of~$\mathrm{K}3^{[2]}$-type over $B$. The blow-up of $\mathcal{K}$ along $\mathcal{Z}$ is a smooth family $\widetilde{\mathcal{K}}$ over~$B$, and the action of $G$ extends to a fibrewise action on $\widetilde{\mathcal{K}}$. 
 	By Propositions \ref{prop:resolution} and \ref{prop:symplectic}, the quotient $\mathcal{Y}=\widetilde{\mathcal{K}}/G$ is a smooth proper family of hyper-K\"{a}hler manifolds over $B$, with fibre over the point $b\in B$ the manifold $Y_{\mathcal{K}_b}$.
 	It thus suffices to find a single $K$ of $\mathrm{Kum}^3$-type such that $Y_K$ is hyper-K\"{a}hler of $\mathrm{K}3^{[3]}$-type.
 	
 	Consider the sixfold $K_J(v_3)$ introduced in Section \ref{sec:specialCase}, which is constructed from a Beauville-Mukai system on a principally polarized abelian surface $J$. By Theorem~\ref{thm:main}, in this case $Y_{K_J(v_3)}$ is birational to the Hilbert scheme $\Km(J)^{[3]}$ on the Kummer K3 surface associated to $J$. 
 	By~\cite[Theorem 4.6]{Huy99}, birational hyper-K\"{a}hler manifolds are deformation equivalent, and therefore $Y_{K_J(v_3)}$ is of $\mathrm{K}3^{[3]}$-type.
\end{proof}

\subsection{The associated $\mathrm{K}3$ surface}\label{subsec:associateK3}
The following computation is entirely analogous to \cite[\S3]{floccari22}. If $Z$ is a compact K\"{a}hler manifold, we let $H^2_{\mathrm{tr}}(Z, \ZZ)$ denote the smallest sub-Hodge structure of $H^2(Z, \ZZ)$ whose complexification contains $H^{2, 0}(Z)$. Let $K$ be a manifold of $\mathrm{Kum}^3$-type, and denote by $\widetilde{K}$ the blow-up of $K$ along $\bigcup_{1\neq g\in G} K^g$. 
By Proposition \ref{prop:resolution} we have a commutative diagram
\[
\begin{tikzcd}
	& \widetilde{K} \arrow[swap]{ld}{p'} \arrow{rd}{q'}\\
	K \arrow[swap]{rd}{q} \arrow[dashed]{rr}{r} && Y_K \arrow{dl}{p}\\
	 & K/G
\end{tikzcd}
\] 
where $p, p'$ are blow-up maps and $q, q'$ are the quotient maps for the action of $G$. 

The pull-back gives an isomorphism ${p'}^*\colon H^2_{\mathrm{tr}}(K, \ZZ) \to H^2_{\mathrm{tr}}(\widetilde{K}, \ZZ)$. Since $G$ acts trivially on $H^2(K, \ZZ)$, the pushforward $q'_*\colon H^2_{\mathrm{tr}}(\widetilde{K}, \ZZ)\to H^2_{\mathrm{tr}}(Y_K,\ZZ)$ is injective, and ${q'}^*q'_*$ is multiplication by $2^5$ on $H^2_{\mathrm{tr}}(\widetilde{K}, \ZZ)$. It follows that 
\[
r_*\colon H^2_{\mathrm{tr}}(K, \ZZ) \to H^2_{\mathrm{tr}}(Y_K, \ZZ)
\]
becomes an isomorphism of Hodge structures after tensoring with $\QQ$. 
Denote by $q_K$ and $q_{Y_K}$ the Beauville-Bogomolov forms on $H^2(K,\ZZ)$ and $H^2(Y_K,\ZZ)$ respectively.

\begin{lemma}\label{lem:intersectionForm}
	For any $x\in H_{\mathrm{tr}}^2(K,\ZZ)$ we have $ q_{Y_K} (r_*(x), r_*(x)) = 2^9 q_K(x, x) $.  
	Therefore,  
	\[
	\frac{1}{16}r_*\colon  H^2_{\mathrm{tr}}(K, \QQ)(2) \xrightarrow{ \ \sim \ } H^2_{\mathrm{tr}}(Y_K, \QQ),
	\]
	is a rational Hodge isometry, where $(2)$ indicates that the form is multiplied by $2$.
\end{lemma}
\begin{proof}
	Let $c_K$ and $c_Y$ be the Fujiki constants of $K$ and $Y$ respectively (\cite{fujiki1987}). This means that we have $\int_K x^6 = c_K q_K(x,x)^3$ for any $x\in H^2(K, \ZZ)$, and similarly for $Y_K$. 
	Let $x\in H_{\mathrm{tr}}^2(K, \ZZ)$; since $x$ is $G$-invariant, we have $r^*(r_* x) = 2^5 x$. We compute:
	\begin{align*}
		q_{Y_K}(r_*x, r_*x)^3 & = \frac{1}{c_{Y_K}}\int_{Y_K} (r_*x)^6 \\
		&
		 = \frac{1}{2^5 c_{Y_K}} \int_K (r^*r_* x)^6 
		 \\
		& = \frac{1}{2^5 c_{Y_K}} \int_K (2^5 x)^6 \\
		& =\frac{2^{25}c_K }{c_{Y_K}} q_K(x,x)^3.
	\end{align*}
	By \cite{rapagnetta2008beauville}, we have $c_{Y_K}= 15$ and $c_K=60$. Hence $q_{Y_K}(r_*x, r_*x)=2^{9} q_K(x,x)$, so $r_*/2^4$ multiplies the form by $2$ and yields the claimed rational Hodge isometry. 
	\end{proof}

We now prove that every projective $K$ of $\mathrm{Kum}^3$-type has a naturally associated K3 surface.

\begin{proof}[Proof of Theorem \ref{thm:AssociatedK3}]
	Let $K $ be a projective variety of $\mathrm{Kum}^3$-type and let $Y_K$ be the crepant resolution of $K/G$. By the above lemma, the transcendental lattice $H^2_{\mathrm{tr}}(Y_K,\ZZ)$ is an even lattice of signature $(2, k)$, and rank at most $6$. By \cite[Corollary 2.10]{Morrison1984}, there exists a $\mathrm{K}3$ surface $S_K$ such that $H^2_{\mathrm{tr}}(S_K,\ZZ) $ is Hodge isometric to $H^2_{\mathrm{tr}}(Y_K,\ZZ)$. A criterion due to independently Mongardi-Wandel \cite{mongardi2015} and Addington \cite[Proposition~4]{Addington2016} ensures that $Y_K$ is birational to $M_{S_{K},H}(v)$, for some primitive Mukai vector $v$ and a $v$-generic polarization $H$ on $S_K$. 
	The surface $S_K$ is uniquely determined up to isomorphism, because	two $\mathrm{K}3$ surfaces of Picard rank at least $12$ with Hodge isometric transcendental lattices are isomorphic, see \cite[Chapter 16, Corollary~3.8]{huyK3}
\end{proof}

\begin{remark}\label{rmk:4families}
	The K3 surfaces $S_K$ come in countably many $4$-dimensional families. A projective $K_0$ of $\mathrm{Kum}^3$-type with Picard rank $1$ gives such a family, consisting of the K3 surfaces $S$ such that $H^2_{\mathrm{tr}}(S,\ZZ)$ is a sublattice of $H^2_{\mathrm{tr}}(S_{K_0}, \ZZ)$.
\end{remark}

Up to isogeny, the K3 surfaces obtained are easily characterized as follows. Recall that $\Lambda_{\mathrm{Kum}^3}=U^{\oplus 3}\oplus \langle -8\rangle$ is the lattice $H^2(K,\ZZ)$ for $K$ of $\mathrm{Kum}^3$-type.

\begin{lemma}\label{lem:upToIsogeny}
	Let $S$ be a projective $\mathrm{K}3$ surface. The following are equivalent:
	\begin{itemize}
		\item there exists an isometric embedding of $H^2_{\mathrm{tr}}(S,\QQ)$ into $\Lambda_{\mathrm{Kum}^3}(2)\otimes_{\ZZ} \QQ$;
		\item there exist a projective variety $K$ of $\mathrm{Kum}^3$-type with associated $\mathrm{K}3$ surface $S_K$ and a rational Hodge isometry $H^2_{\mathrm{tr}}(S, \QQ) \xrightarrow{\ \sim \ } H^2_{\mathrm{tr}}(S_K, \QQ)$. 
	\end{itemize}
\end{lemma}
\begin{proof}
	The second assertion implies the first thanks to Lemma \ref{lem:intersectionForm}.
	Conversely, let $\Phi\colon H^2_{\mathrm{tr}}(S,\QQ) \hookrightarrow \Lambda_{\mathrm{Kum}^3}(2)\otimes_{\ZZ} \QQ$ be an isometric embedding. Choose a primitive sublattice $T\subset \Lambda_{\mathrm{Kum}^3}$ such that $T(2)\otimes_{\ZZ} \QQ$ coincides with the image of $\Phi$. Equip $T$ with the Hodge structure induced by that on $H^2_{\mathrm{tr}}(S, \QQ)$ via $\Phi$. 
	By the surjectivity of the period map \cite{Huy99}, there exists a manifold $K$ of $\mathrm{Kum}^3$-type such that $H^2_{\mathrm{tr}}(K, \ZZ)$ is Hodge isometric to $T$, and Lemma \ref{lem:intersectionForm} gives a Hodge isometry $H^2_{\mathrm{tr}}(S,\QQ)\xrightarrow{\ \sim \ } H^2_{\mathrm{tr}}(S_K,\QQ)$. Since the signature of $T$ is necessarily $(2, k)$ with $k\leq 4$, Huybrechts' projectivity criterion \cite{Huy99} implies that $K$ is projective.
\end{proof}

\section{Applications to the Hodge conjecture}\label{sec:KSHodge}

In this section we prove Theorems \ref{thm:application1} and \ref{thm:application2}. Throughout, all cohomology groups are taken with rational coefficients, which are thus suppressed from the notation.

We start with a simple observation. Let $X, Y$ be smooth projective varieties and let $\phi\colon H^2(X)\to H^2(Y)$ be a morphism of Hodge structures.

\begin{lemma}
	The morphism $\phi$ is induced by an algebraic correspondence if and only if its restriction to $H^{2}_{\mathrm{tr}}(X)$ is so.
\end{lemma}
\begin{proof}
There is a decomposition $H^2(X)=H^2_{\mathrm{tr}}(X) \oplus H^2_{\mathrm{alg}}(X)$, where $H^2_{\mathrm{alg}}(X)$ is spanned by cycle classes of divisors. Similarly, $H^2(Y)= H^2_{\mathrm{tr}}(Y) \oplus H^2_{\mathrm{alg}}(Y)$. 
Any morphism $H^2_{\mathrm{tr}}(X)\to H^2_{\mathrm{alg}}(Y)$ or $H^2_{\mathrm{alg}}(X)\to H^2_{\mathrm{tr}}(Y)$ of Hodge structures is trivial. Hence, $\phi$ gives a Hodge class $$\phi_{\mathrm{tr}} \oplus \phi_{\mathrm{alg}} \ \in\   \bigl(H^2_{\mathrm{tr}}(X)^{\vee} \otimes H^2_{\mathrm{tr}}(Y) \bigr) \oplus \bigl(H^2_{\mathrm{alg}}(X)^{\vee} \otimes H^2_{\mathrm{alg}}(Y)\bigr). $$
The lemma follows because the second summand consists of algebraic classes.
\end{proof}

\begin{proof}[Proof of Theorem \ref{thm:application2}]
	Let $K$ and $K'$ be projective varieties of $\mathrm{Kum}^3$-type and assume that $f\colon H^2(K) \xrightarrow{ \ \sim \ } H^2(K')$ is a rational Hodge isometry. By the above lemma, it suffices to show that the component $f_{\mathrm{tr}}\colon H_{\mathrm{tr}}^2(K) \xrightarrow{\ \sim \ } H_{\mathrm{tr}}^2(K')$ is algebraic. 
	
	Let $Y$, $Y'$ be the varieties of $\mathrm{K}3^{[3]}$-type given by Theorem \ref{thm:MainResult} applied to $K$ and~$K'$ respectively, and let $r\colon K\dashrightarrow Y$ and $r'\colon K'\dashrightarrow Y'$ be the corresponding rational maps. They induce isomorphisms of Hodge structures $r_*\colon H^2_{\mathrm{tr}}(K) \xrightarrow{\ \sim \ } H_{\mathrm{tr}}^2(Y)$ and $r'_*\colon H^2_{\mathrm{tr}}(K') \xrightarrow{\ \sim \ } H_{\mathrm{tr}}^2(Y')$. Let $\bar{f}_{\mathrm{tr}}$ be defined as 
	$$\bar{f}_{\mathrm{tr}} \coloneqq r'_* \circ f_{\mathrm{tr}}\circ (r_*)^{-1} \colon H^2_{\mathrm{tr}}(Y) \to H^2_{\mathrm{tr}}(Y') $$ 
	The inverse of $r_*$ is $r^*/2^5$, and similarly for $r'$. It follows that $f_{\mathrm{tr}}$ is algebraic if and only if $\bar{f}_{\mathrm{tr}}$ is so. 
	By Lemma~\ref{lem:intersectionForm}, the map $\bar{f}_{\mathrm{tr}}$ is a rational Hodge isometry, and hence it is algebraic by Markman's theorem in \cite{markmanRational}.
\end{proof}

\subsection{The Kuga-Satake correspondence} 

The Kuga-Satake construction associates an abelian variety to any polarized Hodge structure of $\mathrm{K}3$-type. We briefly recall this construction, referring to \cite{vanGeemen} and \cite[Chapter 4]{huyK3} for more details. 

Let $(V, q)$ be an effective polarized $\QQ$-Hodge structure of weight $2$ with $h^{2,0}(V)=1$. 
The Clifford algebra $C(V)$ is defined as the quotient
\[
C(V)\coloneqq \bigoplus_{k\geq 0} V^{\otimes k}/\langle v\otimes v - q(v,v) 1 \rangle_{v\in V}.
\]
As a $\QQ$-vector space, $C(V)\cong \bigwedge^{\bullet}(V)$; hence $\dim (C(V))=2^{\mathrm{rk}(V)}$.
The Clifford algebra is $\ZZ/2\ZZ$-graded, $C(V)=C^+(V)\oplus C^-(V)$.
In \cite{deligne1971conjecture}, Deligne shows that the Hodge structure of $V$ induces a Hodge structure of weight $1$ on $C^+(V)$ of type $(1,0), (0,1)$, which admits a polarization and hence defines an abelian variety~$\mathrm{KS}(V)$ up to isogeny. This $2^{\mathrm{rk}(V)-2}$-dimensional abelian variety is called the Kuga-Satake variety of $V$. 
Upon fixing some $v_0\in V$, the action of $V$ on $C(V)$ via left multiplication induces an embedding of weight $0$ rational Hodge structures
\[
V(1) \hookrightarrow \End(C^+(V)) = H^1(\mathrm{KS}(V))^{\vee}\otimes H^1(\mathrm{KS}(V))
\]
which maps $v$ to the endomorphism $w\mapsto vwv_0$.

\begin{remark}\label{rmk:trivialRemarks}
	If $V'\subset V$ is a sub-Hodge structure such that ${V'}^{\bot} $ consists of Hodge classes, then $\mathrm{KS}(V)$ is isogenous to a power of $\mathrm{KS}(V')$. Replacing the form $q$ with a non-zero rational multiple results in isogenous Kuga-Satake varieties. See \cite[Chapter~4, Example 2.4 and Proof of Proposition 3.3]{huyK3}
\end{remark}

Let $X$ be a projective hyper-K\"{a}hler variety of dimension $2n$. The Kuga-Satake variety $\mathrm{KS}(X)$ of $X$ is the abelian variety obtained from $(H^2_{\mathrm{tr}}(X), q_X)$ via the Kuga-Satake construction, where $q_X$ is the restriction of the Beauville-Bogomolov form.
Identifying $H^1(\mathrm{KS}(X))$ with its dual, there exists an embedding of Hodge structures of~$H^2_{\mathrm{tr}}(X)$ into $H^2(\mathrm{KS}(X)\times \mathrm{KS}(X) )$. 
According to the Hodge conjecture this embedding should be induced by an algebraic cycle.

\begin{conjecture}[Kuga-Satake Hodge conjecture] \label{conj:kugaSatake}
	Let $X$ be a projective hyper-K\"{a}hler variety. There exists an algebraic cycle $\zeta$ on $X\times \mathrm{KS}(X) \times \mathrm{KS}(X)$ such that the associated correspondence induces an embedding of Hodge structures 
	$$\zeta_*\colon H^2_{\mathrm{tr}}(X) \hookrightarrow H^2(\mathrm{KS}(X)\times \mathrm{KS}(X) ).$$
\end{conjecture} 

\begin{remark}
	For a K3 surface $X$, the above form of the conjecture is equivalent to \cite[\S10.2]{vanGeemen}. To see this, let $\mathsf{Mot}$ be the category of Grothendieck motives over~$\CC$, in which morphisms are given by algebraic cycles modulo homological equivalence, see \cite{andre}. The motive $\h(X)\in \mathsf{Mot}$ of $X$ decomposes as the sum of its transcendental part~$\h_{\mathrm{tr}}(X)$ and some motives of Hodge-Tate type.
	Since the standard conjectures hold for $X$ and $\mathrm{KS}(X)$, the tensor subcategory of $\mathsf{Mot}$ generated by their motives is abelian and semisimple \cite[Theorem~4.1]{Ara06}. 
	Therefore, Conjecture~\ref{conj:kugaSatake} and the formulation of van Geemen \cite[\S10.2]{vanGeemen} are both equivalent to~$\h^2_{\mathrm{tr}}(X)$ being a direct summand of $\h^2(\mathrm{KS}(X)\times\mathrm{KS}(X))$ in the category $\mathsf{Mot}$.
\end{remark}

We will use the following easy lemma.
\begin{lemma}\label{lem:KSisogeny}
	Let $X$ and $Z$ be projective hyper-K\"{a}hler varieties. Assume that there exists an algebraic cycle $\gamma$ on $Z\times X$ which induces a rational Hodge isometry
	\[
	\gamma_*\colon H^2_{\mathrm{tr}}(Z)\xrightarrow{\ \sim \ } H^2_{\mathrm{tr}}(X)(k),
	\]
	for some non-zero $k\in\QQ$. Then, if Conjecture \ref{conj:kugaSatake} holds for $X$, it holds for $Z$ as well.
\end{lemma}
\begin{proof}
	By Remark \ref{rmk:trivialRemarks} there exists an isogeny $\phi \colon \mathrm{KS}(X)\to\mathrm{KS}(Z)$ of Kuga-Satake varieties. It induces an isomorphism 
	\[
	\phi_*\colon H^2(\mathrm{KS}(X)\times \mathrm{KS}(X)) \xrightarrow{\ \sim \ } H^2(\mathrm{KS}(Z)\times \mathrm{KS}(Z))
	\]
	of rational Hodge structures.
	If Conjecture \ref{conj:kugaSatake} holds for $X$, there exists an algebraic cycle $\zeta$ on $X\times \mathrm{KS}(X) \times \mathrm{KS}(X)$ giving an embedding of Hodge structures 
	\[
	\zeta_*\colon H^2_{\mathrm{tr}}(X) \hookrightarrow H^2(\mathrm{KS}(X)\times \mathrm{KS}(X)).
	\]
 	It follows that the embedding of Hodge structure given by the composition
	\[
	\phi_* \circ \zeta_* \circ \gamma_* \colon H^2_{\mathrm{tr}}(Z) \hookrightarrow H^2(\mathrm{KS}(Z)\times \mathrm{KS}(Z))
	\]
	is induced by an algebraic cylce on $Z\times \mathrm{KS}(Z)\times \mathrm{KS}(Z)$.
\end{proof}

For convenience of the reader we restate Theorem \ref{thm:application1}.

\begin{theorem}\label{thm:KSHC}
	Let $S$ be a projective $\mathrm{K}3$ surface such that there exists an isometric embedding of $H^2_{\mathrm{tr}}(S)$ into $\Lambda_{\mathrm{Kum}^3}(2)\otimes_{\ZZ} \QQ$. 
	Then Conjecture \ref{conj:kugaSatake} holds for $S$.
\end{theorem}

\begin{proof}
	By Lemma \ref{lem:upToIsogeny}, there exists a projective variety $K$ of $\mathrm{Kum}^3$-type with associated K3 surface $S_K$ and a rational Hodge isometry $t_0\colon H^2_{\mathrm{tr}}(S)\xrightarrow{\ \sim \ } H^2_{\mathrm{tr}}(S_K)$. 
	Denote by $Y_K$ the crepant resolution of $K/G$ given by Theorem \ref{thm:MainResult} and let $v, H$ be given by Theorem \ref{thm:AssociatedK3}, so that $Y_K$ is birational to the moduli space $M_{S_K,H}(v)$. Denote by $r\colon K\dashrightarrow Y_K$ the natural rational map of degree $2^5$.
	
	By \cite{Buskin}, \cite{huybrechtsMotives}, the isometry $t_0$ is algebraic. 
	Next, by \cite{mukai1987moduli}, there exists a quasi-tautological sheaf $U $ over $ S_K \times M_{S_K,H}(v)$, which means that there exists an integer~$\rho$ such that for any $F\in M_{S_K,H}(v)$ the restriction of $U$ to $S_K\times \{F\}$ is $F^{\oplus \rho}$. Consider the algebraic class $$\gamma\coloneqq \frac{1}{\rho}\mathrm{ch}(U)\cdot \mathrm{pr}^* \sqrt{\mathrm{td}_{S_K}}\in H^{\bullet}(S_K\times M_{S_K, H}(v)),$$ 
	where $\mathrm{pr}\colon S_K\times M_{S_K, H}(v) \to S_K$ is the projection. By \cite{O'G97}, its K\"{u}nneth component $\gamma_3\in H^6 (S_K\times M_{S_K, H}(v))$ induces a Hodge isometry
	$$t_1\colon H^2_{\mathrm{tr}}(S_K) \xrightarrow{\ \sim \ } H^2_{\mathrm{tr}}(M_{S_K, H}(v)).$$ 
	By \cite{CM13}, the standard conjectures holds for $M_{S_K, H}(v)$ (this also follows from~\cite{Bue18} via the arguments of \cite{Ara06}). Hence, all K\"{u}nneth components of $\gamma$ are algebraic. 
	Let now $f\colon M_{S_K, H}(v)\dashrightarrow Y_K$ be a birational map. Then the pushforward $f_*\colon H^2_{\mathrm{tr}}(M_{S_K, H}(v))\xrightarrow{\ \sim \ } H^2_{\mathrm{tr}}(Y_K)$ is a Hodge isometry. By Lemma \ref{lem:intersectionForm}, a multiple of the composition $r^* \circ f_*$ gives a Hodge isometry $t_2\colon H^2_{\mathrm{tr}}(M_{S_K, H}(v)) \xrightarrow{\ \sim \ } H^2_{\mathrm{tr}}(K)(2)$.

 	It follows that the Hodge isometry 
	\[
	t_2\circ t_1\circ t_0 \colon H^2_{\mathrm{tr}}(S)\xrightarrow{\ \sim \ } H^2_{\mathrm{tr}}(K)(2)
	\]
	is induced by an algebraic cycle on $S\times K$. The Kuga-Satake Hodge conjecture holds for $K$ by \cite{voisinfootnotes}. By Lemma \ref{lem:KSisogeny}, Conjecture~\ref{conj:kugaSatake} holds for the K3 surface~$S$ as well.
\end{proof}	

For the general projective variety $ K$ of $\mathrm{Kum}^3$-type, $H^2_{\mathrm{tr}}(K,\QQ)$ is a rank $6$ Hodge structure with
Hodge numbers $(1, 4, 1)$ and the explicit knowledge of the quadratic form allows one to show
that the Kuga Satake variety is $A^4$, where $A$ is an abelian fourfold of Weil type (cf. \cite[Theorem 9.2]{vanGeemen}). O'Grady moreover shows in \cite{O'G21} that $H^1(A,\QQ)\cong H^3(K,\QQ)$. This is a crucial ingredient in the subsequent work of Markman \cite{markman2019monodromy} and Voisin \cite{voisinfootnotes}.
O'Grady's theorem further allows us to apply a result of Varesco \cite{varesco} to prove the Hodge conjecture for all powers of the K3 surfaces appearing in Theorem \ref{thm:application1}.

\begin{corollary}\label{cor:HCpowers}
	Let $S$ be a projective $\mathrm{K}3$ surface such that there exists an isometric embedding of $H^2_{\mathrm{tr}}(S)$ into $\Lambda_{\mathrm{Kum}^3}(2)\otimes_{\ZZ} \QQ$. Then the Hodge conjecture holds for all powers of $S$.
\end{corollary}

\begin{proof}
	By Remark \ref{rmk:4families}, there exists a $4$-dimensional family $\mathcal{S}\to B$ of projective K3 surfaces of general Picard rank $16$ such that: for each $b\in B$ the fibre $\mathcal{S}_b$ is of the form~$S_{K_b}$ for some $K_b$ of $\mathrm{Kum}^3$-type and for some $0\in B$ there exists a rational Hodge isometry $H^2_{\mathrm{tr}}(S) \xrightarrow{\ \sim \ } H^2_{\mathrm{tr}}(\mathcal{S}_0)$. By \cite{huybrechtsMotives}, it is sufficient to prove the Hodge conjecture for all powers of $\mathcal{S}_0$.
	
	To this end, it will be enough to check that $\mathcal{S}\to B$ satisfies the two assumptions of \cite[Theorem 0.2]{varesco}. As explained in [\textit{loc.\ cit.}, Theorem 4.1], the first of these assumptions is that the Kuga-Satake variety of a general K3 surface in the family is isogenous to $A^4$ for an abelian fourfold $A$ of Weil type with trivial discriminant. This holds by O'Grady's theorem in \cite{O'G21} as already mentioned. 
	The other assumption is that the Kuga-Satake Hodge conjecture holds for the surfaces $\mathcal{S}_b$ for all $b\in B$, which is the content of Theorem \ref{thm:KSHC}.	
\end{proof}

\begin{remark}
	The conclusions of Theorem \ref{thm:KSHC} and Corollary \ref{cor:HCpowers} were known for the two $4$-dimensional families of K3 surfaces studied in \cite{paranjape} and \cite{ILP}. 
	The general fibre has trascendental lattice $U^{\oplus 2}\oplus\langle -2\rangle^{\oplus 2}$ in the first case and $U^{\oplus 2} \oplus \langle -6\rangle \oplus \langle -2\rangle$ in the second. The reader may check that the K3 surfaces studied in \cite{paranjape} satisfy the assumption of our Theorem \ref{thm:KSHC}, while those appearing in \cite{ILP} do not.
\end{remark}

\bibliographystyle{smfplain}
\bibliography{bibliographyNoURL}{}
\end{document}